\documentclass[10pt,a4paper]{amsart}
\pdfoutput=1

\usepackage{pdfpages}
\usepackage[pagebackref]{hyperref}
\usepackage{amsmath,amsfonts,amssymb,amsthm}
\usepackage{mathtools}
\usepackage[all]{xy}
\usepackage{enumitem}
\usepackage{url}
\usepackage[abbrev,alphabetic]{amsrefs}
\usepackage[capitalise]{cleveref}

\usepackage{etoolbox}
\apptocmd{\sloppy}{\hbadness 10000\relax}{}{}

\usepackage{etoolbox}
\apptocmd{\sloppy}{\hbadness 10000\relax}{}{}

\newtheorem{Question}{Question}
\newtheorem{lemm}{Lemma}[section]
\crefname{lemm}{Lemma}{Lemmata}
\newtheorem{theo}[lemm]{Theorem}
\crefname{theo}{Theorem}{Theorems}
\newtheorem{prop}[lemm]{Proposition}

\newtheorem{coro}[lemm]{Corollary}

\newtheorem*{Conjecture*}{Conjecture}
\newtheorem*{example*}{Example}
\newtheorem*{Claim*}{Claim}

\newtheorem*{theo*}{Theorem}
\newtheorem*{prop*}{Proposition}
\newtheorem*{coro*}{Corollary}
\newtheorem*{Question*}{Question}
\newcounter{claimcounter}
\numberwithin{claimcounter}{lemm}

\theoremstyle{definition}

\newcommand{\newqedtheorem}[1]{%
  \newenvironment{#1}
    {\pushQED{\qed}\csname inner@#1\endcsname}
    {\popQED\csname endinner@#1\endcsname}%
  \newtheorem*{inner@#1}%
}
\theoremstyle{remark}
\newqedtheorem{rem*}{Remark}

\newcommand{\newqednumtheorem}[1]{%
  \newenvironment{#1}
    {\pushQED{\qed}\csname inner@#1\endcsname}
    {\popQED\csname endinner@#1\endcsname}%
  \newtheorem{inner@#1}%
}
\newqednumtheorem{exam}[lemm]{Example}
\newqednumtheorem{rem}[lemm]{Remark}

\DeclareMathOperator{\im}{Im}
\DeclareMathOperator{\rk}{rk}

\DeclareMathOperator{\kdim}{K.dim}
\DeclareMathOperator{\lkdim}{l.K.dim}

\DeclareMathOperator{\ann}{ann}
\DeclareMathOperator{\End}{End}
\DeclareMathOperator{\Diag}{Diag}

\DeclareMathOperator{\Mat}{Mat}
\DeclareMathOperator{\reg}{reg}

\DeclareMathOperator{\id}{id}

\DeclareMathOperator{\sg}{sgn}

\newcommand{\ZZ}{\mathbb{Z}}

\newcommand{\RR}{\mathbb{R}}

\newcommand{\PP}{\mathbb{P}}

\newcommand{\DC}{\mathcal D}

\newcommand{\OC}{\mathcal O}
\newcommand{\PC}{\mathcal P}

\newcommand{\UC}{\mathcal U}

\newcommand{\mf}{\mathfrak m}
\newcommand{\nf}{\mathfrak n}

\newcommand*{\doi}[1]{\href{https://doi.org/#1}{doi:#1}}

\date{December 2020}
\title[On the space of Sylvester matrix rank functions]{On the space of Sylvester matrix rank functions}

\author{Andrei Jaikin-Zapirain}
\email{\href{mailto:andrei.jaikin@uam.es}{andrei.jaikin@uam.es}}
\address{Universidad Aut\'onoma de Madrid, Ciudad Universitaria de Cantoblanco, s/n, 28049 Madrid, Spain.}

\author{Diego L\'opez-\'Alvarez}
\email{\href{mailto:diego.lopez@icmat.es}{diego.lopez@icmat.es}}
\address{Instituto de Ciencias Matem\'aticas (ICMAT), Nicol\'as Cabrera 13-15, Campus de Cantoblanco UAM, 28049 Madrid, Spain.}

\begin{document}

\begin{abstract}  Given a ring $R$, the notion of Sylvester rank function was conceived within the context of Cohn's classification theory of epic division $R$-rings.

In this paper we study and describe the space of Sylvester rank functions on certain families of rings, including Dedekind domains, simple left noetherian rings and skew Laurent polynomial rings $\DC[t^{\pm 1};\tau]$ for any division ring $\DC$ and any automorphism $\tau$ of $\DC$. 
 \end{abstract}
 \maketitle
\tableofcontents

\section*{Introduction}

Sylvester matrix rank functions, originally introduced by P. Malcolmson under the name \emph{algebraic rank functions} in \cite{Malcolmson1980}, made their first appearance within the framework of P.M. Cohn's theory of epic division rings. 

Given a ring $R$, an \emph{epic division $R$-ring} is a pair ($\DC$, $\varphi$) where $\DC$ is a division ring and $\varphi$ is a ring epimorphism from $R$ to $\DC$ (equivalently, $\varphi(R)$ generates $\DC$ as a division ring). Two epic division $R$-rings ($\DC_1, \varphi_1$) and $(\DC_2,\varphi_2)$ are said to be \emph{isomorphic} if there exists a ring isomorphism $\varphi: \DC_1\rightarrow \DC_2$ respecting the $R$-structure, i.e., such that the following diagram commutes
$$
 \xymatrix{
        & \ar@{->}[ld]_{\varphi_1} R \ar@{->}[rd]^{\varphi_2} &  \\
        \DC_1 \ar@{->}[rr]^{\varphi} & & \DC_2. 
 }
$$
Note that if $R$ is commutative, any epic division $R$-ring $(\DC,\varphi)$ must also be commutative by epicity. Moreover, it is uniquely determined by the elements in the prime ideal $\ker \varphi$, since we can recover $\DC$ (up to $R$-isomorphism) by first localizing $R$ at $\ker \varphi$ and then taking the residue field. Thus, there is a bijective correspondence between prime ideals of $R$ and epic division $R$-rings (up to $R$-isomorphism).

When $R$ is noncommutative, an epic division $R$-ring may not exist at all (for instance, if $R$ does not have invariant basis number), and even if one exists it is not necessarily determined by the kernel of the map, since there are domains which can be embedded in many different non-isomorphic epic division rings (cf. \cite{Fisher1971}). However, P.M. Cohn realized that an epic division $R$-ring $(\DC,\varphi)$ is determined by the set $\PC$ of square matrices over $R$ mapping to singular matrices over $\DC$, and that this set somehow behaves as a prime ideal. As before, there is a bijective correspondence between these sets $\PC$ and epic division $R$-rings (up to $R$-isomorphism).

One can rephrase this result in the language of Sylvester rank functions as follows.

\begin{theo*}[P.M. Cohn, P. Malcolmson]
  Given a ring $R$, there exists a bijective correspondence between integer-valued Sylvester rank functions on $R$ and epic division $R$-rings up to $R$-isomorphism.
\end{theo*}

Since then, Sylvester rank functions have proven useful for both their role as a classifying tool and as a language, and have been used in a variety of contexts. For instance, to study homomorphisms to simple artinian rings (\cite{Schofield1985}), to tackle the direct finiteness conjecture (\cite{AOP2002}), to study representations of finite dimensional algebras \cite{Elek2017}, and in recent advances regarding the strong Atiyah conjecture, the L\"uck approximation conjecture and the study of universal division rings of fractions
(\cite{Jaikin2019A}, \cite{Jaikin2019S}, \cite{Jaikin2020A}, \cite{Jaikin2020B}, \cite{Jaikin2019B}, \cite{JL2020}). 

The main concern of this paper is that, although for the previous reasons they have become interesting as objects themselves (cf. \cite{AC2019}, \cite{Li2019}, \cite{JiLi2020}), little is known in general about the space $\PP(R)$ of all the Sylvester rank functions that can be defined on a given ring $R$. In particular, the question that motivated the description of $\PP(R)$ for the families of rings presented here (essentially, polynomial rings) was the following (cf. \cite{Jaikin2019S}*{Question~8.7}).

\begin{Question} \label{Question}
  Let $R$ be a simple von Neumann regular ring with a Sylvester rank function $\rk$ such that $R$ is $\rk$-complete. Is it true that every Sylvester rank function on $Z(R)[t]$ extends uniquely to a Sylvester rank function on $R[t]$?
\end{Question}

Here, $\rk$-complete means that $\rk(a)$ is non-zero for non-zero elements $a$ of $R$ and that $R$ is complete with respect to the metric $\delta_{\rk}$ given by $\delta_{\rk}(x,y) = \rk(x-y)$. Although we will not make use of this fact, the notion of Sylvester matrix rank function on a von Neumann regular ring coincides with the notion of pseudo-rank function presented in \cite{Goodearl1991}. In particular, under the above hypothesis, it can be shown that $\mathbb P(R) = \{\rk\}$ (\cite{Goodearl1991}*{Proposition~19.13 and Theorem~19.14}).

This question arose during a first attempt of the first author to prove the transcendental inductive step in his proof of the sofic L\"uck approximation conjecture in \cite{Jaikin2019A}, and we shall give a positive answer for the particular case of simple artinian rings.

The structure of the paper is the following. In \cref{sect:Sylvester} we recall the notion of Sylvester rank function and state its basic structural properties. In \cref{sect:leftartinianprimary} we start by studying the space of rank functions on a certain subfamily of left artinian primary rings. This subfamily appears when dealing with quotients of Dedekind domains and skew Laurent polynomial rings $\DC[t^{\pm 1};\tau]$, where $\DC$ is a division ring and $\tau$ an automorphism of $\DC$, and therefore allows us to get a partial picture of the space of Sylvester rank functions for those two families, which are later studied in \cref{sect:Dedekind} and \cref{sect:Laurent}, respectively. In \cref{sect:simple_noetherian} we discuss the case of simple left noetherian rings, which also appear naturally when dealing with skew Laurent polynomial rings.

\section*{Acknowledgements}
The present work was partially supported by the grant MTM2017-82690-P of the Spanish MINECO and by the ICMAT Severo Ochoa project SEV-2015-0554, and it is part of the second author PhD project at the Autonomous University of Madrid (UAM). 

The authors want to thank Hanfeng Li and Javier S\'anchez for useful comments on a preliminary version of the document. The second author also wants to thank Carlos Abad for helpful discussions. 

\section{Sylvester rank functions and first examples} \label{sect:Sylvester}

In this section we introduce and relate the notions of Sylvester matrix and module rank function, and present some basic examples of rings for which we can describe the space of Sylvester rank functions. 

A \emph{Sylvester matrix rank function} $\rk$ on a ring $R$ is a map that assigns a non-negative real number to each matrix over $R$ and that satisfies the following.
 \begin{enumerate}
\item [(SMat1)] $\rk(A)=0$ if $A$ is any zero matrix and $\rk(1)=1$;
\item [(SMat2)]  $\rk(AB) \le \min\{\rk(A), \rk(B)\}$ for any matrices $A$ and $B$ which can be multiplied;
\item[(SMat3)] $\rk\left (\begin{array}{cc} A & 0\\ 0 & B\end{array}\right ) = \rk(A) + \rk(B)$ for any matrices $A$ and $B$;
\item[(SMat4)] $\rk \left (\begin{array}{cc} A & C\\ 0 & B\end{array}\right ) \ge \rk(A) + \rk(B)$ for any matrices $A$, $B$ and $C$ of appropriate sizes.
\end{enumerate}

Some of the basic properties of Sylvester matrix rank functions are collected in \cite{Jaikin2019A}*{Proposition~5.1}, and for a more extensive survey on the topic the reader may consult \cite{Jaikin2019S}*{Section~5}. For instance, note from (SMat1) and (SMat3) that $\rk(I_n) = n$ and hence, from (SMat2), that $n\times n$ invertible matrices have rank $n$. Moreover, from (SMat2) we also deduce that multiplication by an invertible matrix does not change the rank, since for any invertible $A\in \Mat_n(R)$ and for any $n\times m$ matrix $B$, 
$$
 \rk(B) = \rk(A^{-1}AB) \leq \rk(AB) \leq \rk(B).
$$
Similarly, we can add or delete rows and columns of zeros without changing the rank, since for any $A\in\Mat_{n\times m}(R)$ and integers $k,l\ge 0$,
$$
 \rk\left(\begin{pmatrix} A_{n\times m} & 0_{n\times l}\\ 0_{k\times m} & 0_{k\times l}\end{pmatrix}\right) = \rk\left(\begin{pmatrix}I_n \\ 0_{k\times n} \end{pmatrix} A_{n\times m} \begin{pmatrix} I_m 0_{m\times l}\end{pmatrix} \right) \le \rk(A)
$$
and the other inequality is obtained similarly.

We also need the following two properties. The first one will be used to show that the coefficients appearing in the proof of \cref{prop:Kt_tn} are non-negative.  

\begin{lemm} \label{lemm:positive_coefficients}
 Let $\rk$ be a Sylvester matrix rank function on a ring $R$. Take an element $a\in R$ and set $b_i = \rk(a^i)-\rk(a^{i+1})$. Then, for every $i\ge 0$, $b_i\ge b_{i+1}$. In particular, if $a$ is nilpotent and $a^{n+2}=0$, then $\rk(a^{n})\ge 2\rk(a^{n+1})$.
\end{lemm}
\begin{proof}
From (SMat3), (SMat4) and the previous observation, we obtain:
$$
\begin{matrix*}[l]
 \rk(a^{n+2})+ \rk(a^n) & = & \rk 
 \begin{pmatrix} 
  a^{n+2} & 0 \\
  0 & a^n 
 \end{pmatrix} \medskip \\
  & = & \rk\left(
  \begin{pmatrix} 
  -1 & a \\
  0 & 1 
 \end{pmatrix}
  \begin{pmatrix} 
  a^{n+2} & 0 \\
  0 & a^n 
 \end{pmatrix}
  \begin{pmatrix} 
  0 & 1 \\
  1 & a 
 \end{pmatrix} \right) \medskip \\
  & = & \rk
 \begin{pmatrix} 
  a^{n+1} & 0 \\
  a^n & a^{n+1} 
 \end{pmatrix} \geq 2\rk(a^{n+1}).
\end{matrix*}
$$
\end{proof} 

The second lemma studies additivity of the rank under certain conditions. In particular, when $A$ and $B$ are orthogonal and idempotent, and when $A$ and $B$ are matrices over a cartesian product $R_1 \times R_2$ such that $A\in \Mat_{n\times m}(R_1\times\{0\})$ and $B\in \Mat_{n\times m}(\{0\}\times R_2)$.

\begin{lemm} \label{lemm:additivity}
 Let $\rk$ be a Sylvester matrix rank function on a ring $R$. Let $A,B\in \Mat_{n\times m}(R)$, and assume that there exist $C\in \Mat_{n\times n}(R)$, $D \in \Mat_{m\times m}(R)$ such that $CA = A$, $BD = B$ and $AD = CB = 0$. Then $\rk(A+B) = \rk(A)+\rk(B)$.
\end{lemm}

\begin{proof}
Since the rank is invariant under multiplication by invertible matrices, we have the following. 
  $$
\begin{matrix*}[l]
 \rk(A)+ \rk(B) & = & \rk 
 \begin{pmatrix} 
  A & 0 \\
  0 & B 
 \end{pmatrix} \medskip \\
 & =  & \rk\left(
  \begin{pmatrix} 
  I_n & 0 \\
  I_n & I_n 
 \end{pmatrix}
  \begin{pmatrix} 
  A & 0 \\
  0 & B 
 \end{pmatrix}
  \begin{pmatrix} 
  I_m & 0 \\
  I_m   & I_m 
 \end{pmatrix} \right) \medskip \\
  & =  & \rk\left(
  \begin{pmatrix} 
  0 & I_n \\
  I_n & 0
 \end{pmatrix}
  \begin{pmatrix} 
  I_n & -C \\
  0 & I_n
 \end{pmatrix}
  \begin{pmatrix} 
  A & 0 \\
  A+B & B 
 \end{pmatrix}
  \begin{pmatrix} 
  I_m & -D \\
  0   & I_m 
 \end{pmatrix} \right) \medskip \\
 & = & \rk
 \begin{pmatrix} 
  A+B &   0 \\
  0 & 0 
 \end{pmatrix} = \rk(A+B).
\end{matrix*}
$$
\end{proof}

A \emph{Sylvester module rank function} $\dim$ on $R$ is a map that assigns a non-negative real number to each finitely presented left $R$-module and that satisfies the following.
  \begin{enumerate}
\item [(SMod1)] $\dim \{0\} =0$, $\dim R =1$;
\item [(SMod2)]  $\dim(M_1\oplus M_2)=\dim M_1+\dim M_2$;
\item[(SMod3)] if $M_1\to M_2\to M_3\to 0$ is exact then
$$\dim M_1+\dim M_3\ge \dim M_2\ge \dim M_3.$$
\end{enumerate}

These notions of Sylvester matrix and module rank function were defined by P. Malcolmson in \cite{Malcolmson1980}, and they were shown to be in bijective correspondence. Although he considered only integer-valued rank functions, most of the original proof works without changes for the general case. We fill in the details for the sake of completeness.

\begin{prop}[Malcolmson]\label{prop:Malcolmson}
Let $R$ be a ring. 
\begin{itemize}
   \item[(i)] If $\dim$ is a Sylvester module rank function on $R$, then we can define a Sylvester matrix rank function by assigning to each $A\in \Mat_{n\times m}(R)$, the value
   $$
   \rk(A):= m-\dim(R^m/R^nA).
   $$   
   \item[(ii)] If $\rk$ is a Sylvester matrix rank function on $R$, then we can define a Sylvester module rank function by assigning to any finitely presented left $R$-module with presentation $M = R^m/R^nA$ for some $A\in \Mat_{n\times m}(R)$, the value
   $$
   \dim(M):= m-\rk(A).
   $$
   This value does not depend on the given presentation. 
\end{itemize}
We say in this case that $\rk$ and $\dim$ are \emph{associated}.
\end{prop}

\begin{proof}
 In \cite{Malcolmson1980}*{Theorem~4}, Malcolmson proved the result for integer-valued rank functions. The same proof works in the general case to show that $(i)$ actually defines a Sylvester matrix rank function, and that the definition in $(ii)$ is independent of the presentation. Under the integer-valued hypothesis, the existence of a division $R$-ring associated to $\rk$ allowed him to directly check (SMod1)-(SMod3) in $(ii)$. For the general case, observe that (SMod1) and (SMod2) follow immediately from (SMat1) and (SMat3). We check (SMod3) in two steps.
 
Assume first that we have a surjective homomorphism $g: M' \rightarrow M$, and consider any presentation $M' \cong R^m/R^nA$ of $M'$. Since $M$ is finitely presented, this induces a presentation of $M$ of the form $M \cong R^m/R^kB$ with $R^nA\subseteq R^kB$. Using the lifting property of free modules, there exists a matrix $C$ such that the following commutes
    $$
    \xymatrix{ R^n \ar[d]_{r_C} \ar[r]^{r_A} & R^m \ar[r] \ar[d]^{\id_{R^m}} & M' \ar[d]^{g} \ar[r] & 0 \\ R^k \ar[r]^{r_B} & R^m \ar[r] & M \ar[r] & 0.}
    $$
Thus, $A= CB$ and as a consequence of (SMat2) we get $\dim(M')\geq \dim(M)$.

Suppose now that we have an exact sequence $M_1 \xrightarrow{f} M_2 \xrightarrow{g} M_3 \rightarrow 0$ and presentations 
    $$\begin{matrix}R^n \xrightarrow{r_A} R^m \xrightarrow{p_1} M_1 \rightarrow 0, & R^k \xrightarrow{r_B} R^l \xrightarrow{p_3} M_3 \rightarrow 0 \end{matrix}$$
of $M_1$ and $M_3$, respectively. Then $M_2$ can be $(m+l)$-generated by the image of the generators of $M_1$ defined through $p_1$ and some preimages of the generators of $M_3$ defined through $p_3$. We can construct a surjective homomorphism $\varphi: R^m\oplus R^l\to M_2$ and a homomorphism $\psi: R^n\oplus R^k\to R^m\oplus R^l$ such that the following commutes with exact rows
$$ \xymatrix{ 
0 \ar[r] & R^n \ar[d]_{r_A} \ar[r]^-{\iota_1'} & R^{n}\oplus R^k \ar[r]^-{\pi_2'} \ar[d]^{\psi} & R^k \ar[d]^{r_B} \ar[r] & 0\\
0 \ar[r] & R^m \ar[d]_{p_1} \ar[r]^-{\iota_1} & R^{m}\oplus R^l \ar[r]^-{\pi_2} \ar[d]^{\varphi} & R^l \ar[d]^{p_3} \ar[r] & 0 \\ & M_1 \ar[r]^{f} & M_2 \ar[r]^{g} & M_3 \ar[r] & 0  }
$$ 
and $\im \psi \subseteq \ker \varphi$, where $\iota_1,\iota_1',\pi_2,\pi_2'$ denote the natural embeddings and projections.
In particular, $\im \psi$ is finitely generated and $M_2' := (R^m\oplus R^l)/\im \psi$, which is then finitely presented,  admits a surjection $M_2' \to (R^{m}\oplus R^l)/\ker\varphi \cong M_2$. The first step of the proof tells us that $\dim(M_2')\ge \dim(M_2)$.

We can realize $\psi$ as right multiplication by some $(n+k)\times (m+l)$ matrix $D$, and from the commutativity of the previous diagram we can deduce that $D$ has the form
$$
 D = \begin{pmatrix} A & 0 \\ C & B\end{pmatrix}
$$
for some matrix $C$ of size $m\times k$. Since multiplying by invertible matrices does not change the rank, we get from (SMat4) that
$$
\begin{matrix*}[l]
 \rk(D) = \rk\begin{pmatrix} B & C \\ 0 & A \end{pmatrix} \ge \rk(A)+\rk(B)
\end{matrix*}
$$
As a consequence, $M_2' \cong (R^{m}\oplus R^l)/(R^n\oplus R^k)D$ satisfies 
$$
\dim(M_2') = m+l-\rk(D) \le m+l-\rk(A)-\rk(B) = \dim(M_1)+\dim(M_3).
$$
Adding everything up, we have (SMod3), what finishes the proof. 
\end{proof}

Given a ring $R$, we use $\mathbb{P} (R)$ to denote the space of Sylvester matrix (or equivalently, module) rank functions on $R$, which is a compact convex subset of the space of real-valued functions on matrices over $R$. 

We say that a rank function $\rk\in \mathbb P(R)$ is \emph{extreme} or \emph{an extreme point of $\PP(R)$} if it admits no non-trivial expression as a convex combination of two different elements in $\PP(R)$. For the families of rings that we consider here, the particular structure of finitely generated modules allows us to describe $\mathbb P(R)$ by identifying the extreme ranks, in such a way that any other rank function can be uniquely obtained from those as a (possibly infinite) convex combination.

In order to pursue this goal, we first need strategies to produce rank functions on a ring. For instance, observe that given a ring homomorphism $\varphi: R\rightarrow S$ and a Sylvester matrix rank function $\rk$ on $S$, we can define a rank function $\varphi^{\sharp}(\rk)$ on $R$ by setting $\varphi^{\sharp}(\rk)(A) = \rk(\varphi(A))$ for any matrix $A$ over $R$. Thus, we can get a partial picture of $\mathbb P(R)$ if we study the space of rank functions of proper quotients of $R$, or if we find appropriate homomorphisms to rings in which we have defined a rank function. Moreover, the induced map $\varphi^{\sharp}:\PP(S)\to \PP(R)$ is continuous and convex-linear, i.e., it preserves convex combinations. 

Among the rank functions defined in this way we distinguish those coming from homomorphisms to von Neumann regular rings, i.e., rings $\UC$ in which, for every element $x\in \UC$, there exists an element $y\in \UC$ satisfying $xyx=x$. In this case, we say that the induced rank function on $R$ is \emph{regular}, and we denote by $\mathbb P_{\reg}(R)$ the space of regular Sylvester rank functions on $R$, which is a closed convex subset of $\mathbb P(R)$ (\cite{Jaikin2019A}*{Proposition~5.9}).

When $\varphi$ is epic (i.e. right cancellable as a ring homomorphism), H. Li proved that the induced map $\varphi^{\sharp}: \mathbb P(S) \to \mathbb P(R)$ is injective (\cite{Li2019}*{Theorem~8.1}). This applies, in particular, to Ore localizations $T^{-1}R$ of $R$ with respect to a multiplicative set of non-zero-divisors $T$ satisfying the left Ore condition. In this case (\cite{Jaikin2019S}*{Proposition~5.2}), we can identify 
\begin{equation} \label{Ore}
 \mathbb P(T^{-1}R) = \{\rk\in \mathbb P(R): \rk(t)=1 \textrm{ for every } t\in T \}.
\end{equation}

Another important remark about 
$\mathbb P(R)$ (\cite{Li2019}*{Remark~7.1}) is that the spaces of rank functions of Morita equivalent rings are homeomorphic. We state and prove here the bijective correspondence to make explicit the relation between $\mathbb P(R)$ and $\mathbb P(\Mat_n(R))$, as in \cite{Goodearl1991}*{Corollary~16.10 \& Proposition~16.20}.

\begin{prop} \label{prop:Morita}
  Let $R$, $S$ be Morita equivalent rings. Then there exists a bijective correspondence between $\mathbb P(R)$ and $\mathbb P(S)$ preserving the extreme points. Moreover, if $\iota: R\rightarrow \Mat_n(R)$ is the diagonal embedding, this bijection is convex-linear and sends $\rk\in \mathbb P(\Mat_n(R))$ to the rank function $\iota^{\sharp}(\rk)$ and $\rk'\in \mathbb P(R)$ to the rank function $\frac{1}{n}\rk'$.
\end{prop}

\begin{proof}
 Let $R$-Mod (respectively $S$-Mod) denote the category of left $R$-modules (respectively $S$-modules), and let $F: R\operatorname{-Mod} \rightarrow S\operatorname{-Mod}$, $G: S\operatorname{-Mod}\rightarrow R\operatorname{-Mod}$ be the associated equivalence between these categories with $FG$ and $GF$ naturally equivalent to the corresponding identity functors. Recall that an equivalence of categories preserves direct sums, short exact sequences and finitely presented modules (\cite{Lam1999}*{Section 18A}).

Let $\dim$ be a module rank function on $S$. Since $F(R)$ is a progenerator in $S\operatorname{-Mod}$ (\cite{Lam1999}*{Remark 18.10(A)}), $S$ is a direct summand of $F(R)^n$ for some positive integer $n$. Therefore, from (SMod1) and (SMod2) we obtain that $\dim(F(R))>0$. 

By the previous remarks, the expression $\dim'(N) = \frac{1}{\dim(F(R))}\dim(F(N))$ for an $R$-module $N$ now defines a Sylvester module rank function on $R$, so we have a map $\mathbb P(S)\rightarrow \mathbb P(R)$. Similarly, we can define a map $\mathbb P(R)\rightarrow \mathbb P(S)$. Finally, by the natural equivalence, we have $FG(M)\cong M$ for every $S$-module $M$ and $GF(N)\cong N$ for every $R$-module $N$. As a consequence, one can check that both maps are mutual inverses. 

Finally, they preserve extreme points. For instance, if we have a convex combination $\dim = \lambda\dim_1 + (1-\lambda)\dim_2$ in $\mathbb P(S)$, where $1>\lambda>0$, then we obtain a linear combination
$$
  \dim' = \displaystyle \frac{\lambda\dim_1(F(R))}{\dim(F(R))}\dim'_1+\frac{(1-\lambda)\dim_2(F(R))}{\dim(F(R))}\dim'_2
$$
which is convex, since the coefficients are positive and add up to one.

For the particular case of $S = \Mat_n(R)$, the equivalences of categories are defined on objects (cf. \cite{Lam1999}*{Theorem 17.20}) by
$$
\begin{matrix*}[c]
    F:& R\operatorname{-Mod}& \to & S\operatorname{-Mod} & & G:& S\operatorname{-Mod}& \to & R\operatorname{-Mod} \\
    &P&\mapsto& P^{n\times 1}&, & &Q & \mapsto& E_{11}Q
\end{matrix*}
$$
where $E_{11}$ is the $n\times n$ matrix having $1$ in the upper left corner and zeros everywhere else. Observe that $E_{11}Q \cong R^{1\times n} \otimes_{S} Q$ as $R$-modules and $P^{n\times 1} \cong R^{n\times 1}\otimes_R P$ as $S$-modules.

Let $\rk$ be a Sylvester matrix rank function on $S$ with associated module rank function $\dim$. Since $F(R)^n\cong S$, we have $\dim(F(R))= \frac{1}{n}$. Moreover, for any $A\in \Mat_{k\times l}(R)$, if $M = R^l/R^kA$, there exists an isomorphism of $S$-modules
$$
 \bigoplus_{i = 1}^n F(M) = \bigoplus_{i = 1}^n M^{n\times 1} \cong S^l/S^kB 
$$
where $B = \iota(A)$. Thus, if $\rk'$ is the Sylvester matrix rank function associated to $\dim'$ as defined before, then
$$
\begin{matrix*}[l]
 \rk'(A) &= l-\dim'(M) = \displaystyle l-\frac{\dim(M^{n\times 1})}{\dim(F(R))} = l-n\dim(M^{n\times 1}) =  \\[10pt]
 &=  l-\dim(S^l/S^kB) = \rk(B) = \iota^{\sharp}(\rk)(A),
 \end{matrix*}
$$
so we conclude that $\rk' = \iota^{\sharp}(\rk)$.

Conversely, if $\rk'$ is a Sylvester matrix rank function on $R$ with associated module rank function $\dim'$, and we denote by $\rk$ and $\dim$ the corresponding rank functions on $S$ given by the Morita equivalence, then, from the $R$-module isomorphisms
$$
\begin{matrix*}[c]
 G(S) \cong R^{n} & \textrm{ and } & R^{1\times n}\otimes_S S^l/S^kB \cong R^{nl}/R^{nk}B 
 \end{matrix*}
$$
where $B\in \Mat_{k\times l}(S)$ is considered on the right as an $nk\times nl$ matrix over $R$, we obtain that
$$
\begin{matrix*}[l]
 \rk(B) &= l-\dim(S^l/S^kB) = \displaystyle l-\frac{\dim'(R^{1\times n}\otimes_S S^l/S^kB)}{\dim'(G(S))} =  \\[10pt]
 &= l-\displaystyle \frac{1}{n}\dim'(R^{nl}/R^{nk}B) =  l - \frac{1}{n}(nl-\rk'(B)) = \frac{1}{n}\rk'(B),
 \end{matrix*}
$$
from where $\rk = \frac{1}{n}\rk'$. Since this correspondence preserves convex combinations, this finishes the proof.
\end{proof}

As a first example, consider a division ring $\DC$. Since every finitely generated module $M$ over $\DC$ is isomorphic to some $\DC^n$, then observe from (SMod1) and (SMod2) that, for every $\dim\in \mathbb P(\DC)$, we should have $\dim(M) = n$. Thus, there exists only one Sylvester module rank function on $\DC$, namely, its dimension function $\dim_{\DC}$, which also satisfies (SMod3). The corresponding Sylvester matrix rank function is then $\rk_{\DC}$, the usual rank on matrices over $\DC$.

Now, in view of \cref{prop:Morita}, for every positive integer $n$, there exists a unique Sylvester matrix rank function on $\Mat_n(\DC)$, namely, $\frac{1}{n}\rk_{\DC}$.

We finish the section relating the space of rank functions on a finite cartesian product of rings with the space on every factor. Together with the previous examples, this gives a complete description of the space of rank functions on a semisimple artinian ring. The example of finite cartesian products of von Neumann regular rings (and in particular semisimple artinian rings) was already studied in \cite{Goodearl1991}*{Theorem~16.5 \& Corollary~16.6}.

\begin{rem}
Let $R = R_1\times\cdots\times R_n$ and let $\pi_i:R\to R_i$ denote the natural projections. We say that the rank $\rk\in \PP(R)$ can be uniquely expressed as a convex combination of ranks on $R_i$ if there exist uniquely determined non-negative coefficients $\lambda_1,\dots, \lambda_n$ with $\sum_i \lambda_i = 1$ and, for every $\lambda_i>0$, a uniquely determined $\rk_i\in \PP(R_i)$, such that $\rk = \sum_i \lambda_i \pi_i^{\sharp}(\rk_i)$.

The assumption that $\rk_i$ exists and it is uniquely determined only for $\lambda_i>0$ is needed (cf. \cite{Goodearl1991}*{Theorem~16.5}) to address properly the cases in which some $R_i$ does not admit Sylvester rank functions or that $\lambda_i=0$ and $R_i$ admits more than one Sylvester rank function (in which case the expression would not be unique).
\end{rem}

\begin{prop} \label{prop:directproduct}
  Let $R_1, R_2$ be rings. Any Sylvester rank function on $R = R_1\times R_2$ is a uniquely determined convex combination of Sylvester rank functions on $R_1$ and $R_2$. In particular, the set of extreme points on $\mathbb P(R)$ is the disjoint union of the sets of extreme points of $\mathbb P(R_1)$ and $\mathbb P(R_2)$.
\end{prop}

\begin{proof}
  Let $\pi_i:R\to R_i$ be the natural projections. Since $\pi_i$ is a surjective ring homomorphism, the map $\pi_i^{\sharp}: \mathbb P(R_i)\to \mathbb P(R)$ is injective (\cite{Li2019}*{Theorem~8.1}). Consider also the natural additive maps $\iota_i: R_i\to R$.
  
  Observe that if $A\in \Mat_{n\times m}(R)$, then $A = A_1+A_2$ where $A_1 = \iota_1\pi_1(A)$, $A_2 = \iota_2\pi_2(A)$. Moreover, if $I_{n,1}\in \Mat_n(R)$, $I_{m,2}\in \Mat_m(R)$ denote the diagonal matrices whose entries are all equal to $(1,0)$ and $(0,1)$, respectively, we have
  $$
  \begin{matrix*}
    I_{n,1}A_1 = A_1, & I_{n,1}A_2 = 0, & A_1I_{m,2} = 0, & A_2I_{m,2} = A_2.
  \end{matrix*}
  $$
  Thus, \cref{lemm:additivity} tells us that for any rank function $\rk\in \mathbb P(R)$, we have
  $$\rk(A) = \rk(A_1)+\rk(A_2).$$
  In particular, we obtain that $1 = \rk((1,0))+\rk((0,1))$. Now, if $\rk((1,0)) = 0$, then $\rk(A_1)=\rk(I_{n,1}A_1)\le \rk(I_{n,1}) = n\rk((1,0)) = 0$, and if $\rk((1,0))>0$, then the expression $\rk_1 = \frac{1}{\rk((1,0))}\rk\circ\iota_1$ defines a Sylvester matrix rank function on $R_1$. Similarly, if $\rk((0,1)) = 0$ then $\rk(A_2) = 0$, and we can define a rank function $\rk_2$ on $R_2$ if $\rk((0,1))>0$. One can check that
  $$
   \rk = \rk((1,0))\pi_1^{\sharp}(\rk_1) + \rk((0,1))\pi_2^{\sharp}(\rk_2),
  $$
  where we understand that $\rk_i$ is considered only when the coefficient is non-zero.
  Moreover, if we had another expression $\rk = \lambda\pi_1^{\sharp}(\rk_1') + (1-\lambda)\pi_2^{\sharp}(\rk_2')$
  for some $\rk_i'\in \mathbb P(R_i)$ and $0\le \lambda \le 1$, then $\rk((1,0)) = \lambda$ and $\rk((0,1)) = 1-\lambda$. Hence, for every matrix $B$ over $R_1$,
  $$
   \rk(\iota_1(B)) = \lambda \rk_1(B) = \lambda \rk_1'(B),
  $$
  from where $\rk_1 = \rk_1'$, and similarly $\rk_2 = \rk_2'$. Thus, the combination is unique.
  
  From the previous expression, one can also deduce that if $\rk$ is an extreme point in $\mathbb P(R)$, then either $\rk((1,0))=1$ and $\rk_1$ is an extreme point in $\mathbb P(R_1)$ or $\rk((0,1))=1$ and $\rk_2$ is an extreme point in $\mathbb P(R_2)$. 
  
  Conversely, if, for instance, $\rk_1$ is an extreme point in $\mathbb P(R_1)$, then $\pi_1^{\sharp}(\rk_1)$ is a rank function on $R$ which takes value $1$ on $(1,0)$. Therefore, if we had $\pi_1^{\sharp}(\rk_1) = \lambda\rk+(1-\lambda)\rk'$ with $1>\lambda >0$ and $\rk\ne \rk'$ on $\PP(R)$, then necessarily $\rk((1,0)) = \rk'((1,0)) = 1$ (and consequently $\rk((0,1)) = \rk'((0,1)) = 0$). Thus, reasoning as before, we can see that $\rk\circ\iota_1$ and $\rk'\circ\iota_1$ define rank functions on $R_1$ such that $\rk = \pi_1^{\sharp}(\rk\circ\iota_1)$ and $\rk' = \pi_1^{\sharp}(\rk'\circ\iota_1)$ (in particular, they are different) and 
  $$
   \rk_1 = \lambda \rk\circ\iota_1 + (1-\lambda)\rk'\circ\iota_1,
  $$
  which contradicts the fact that $\rk_1$ is extreme. This finishes the proof.
\end{proof}

An important remark before moving on to the next section is that all the rings we are going to consider from now on are left noetherian, and hence the words finitely presented and finitely generated are interchangeable for left $R$-modules. (cf. \cite{Rotman2009}*{Corollary~3.19})

\section{Left artinian primary rings} \label{sect:leftartinianprimary}

In this section we study the space of Sylvester matrix rank functions on a family of rings which is deeply related to the families in \cref{sect:Dedekind} and \cref{sect:Laurent}, namely, left artinian primary rings. More precisely, we give a description of this space for those left artinian primary rings whose Jacobson radical is generated by a central element. Throughout the section, $J(R)$ denotes the Jacobson radical of the ring $R$.

Following \cite{Pierce1982}, a ring ($\ZZ$-algebra) 
$R$ is \emph{local} if $R/J(R)$ is a division ring, or equivalently, if the set of all non-units in $R$ form a (two-sided) ideal, which is necessarily $J(R)$ (cf. the proof of \cite{Pierce1982}*{Proposition~5.2}). A ring $R$ is \emph{primary} if $R/J(R)$ is simple. 

When $R$ is left (or right) artinian, we can reduce the study of the space of rank functions on the latter family of rings to the study of local rings through \cref{prop:Morita} and the following result (\cite{Pierce1982}*{Proposition~6.5a}).

\begin{prop} \label{prop:Pierce}
 If $R$ is a left artinian primary ring, then there exist a unique $s$ and a unique (up to isomorphism) left artinian local ring $S$ such that $R\cong \Mat_s(S)$. 
\end{prop}

The following example of local artinian ring serves as a motivation for the general treatment.

\subsection{The case of \texorpdfstring{$K[t]/(t^n)$}{K[t]\_tn}} \label{subsect:Kt_tn}

If $K$ is a commutative field and $n$ is a positive integer, then the ring $R = K[t]/(t^n)$ is an example of local artinian ring. As a consequence of the structure of modules on $K[t]$, every finitely generated $R$-module can be expressed as a direct sum of the indecomposable $R$-modules $K[t]/(t^i)\cong R/(t^i+(t^n))$, $1\le i \le n$. Thus, from (SMod2), any Sylvester module rank function on $R$ is determined by its values on these modules, or equivalently, any Sylvester matrix rank function is determined by its values on the elements $t^i+(t^n)$. 

We can define $n$ Sylvester matrix rank functions $\rk_1,\dots, \rk_n$ on $R$ through the canonical homomorphisms $R \rightarrow \End_K(K[t]/(t^k))\cong \Mat_k(K)$ with $p+(t^n)\mapsto \phi_k^p \mapsto A_p$, where $\phi_k^p$ is the endomorphism given by right multiplication by $p+(t^k)$ and $A_p$ is its associated matrix with respect to the canonical basis in $K[t]/(t^k)$ (here, we consider endomorphisms acting on the right, so both are ring homomorphisms). If $\rk_{K}$ denotes the usual rank function on $K$, then the unique rank function on $\Mat_k(K)$ is $\frac{1}{k}\rk_K$, and when we pull it back to $R$ we obtain a regular rank function $\rk_k$ satisfying
$$
\rk_k(t^i+(t^n)) = \begin{cases}
\frac{k-i}{k} & \textrm{if}~ i\leq k \\
0 & \textrm{otherwise}
\end{cases}
$$
Moreover, any other rank function on $R$ is a convex combination of the above ranks.

\begin{prop}\label{prop:Kt_tn}
Let $K$ be a field, $n$ a positive integer and set $R= K[t]/(t^n)$. There exist exactly $n$ extreme points in $\mathbb P(R)$, which are the Sylvester matrix rank functions $\rk_1,\dots, \rk_n$, and any other rank function can be uniquely expressed as a convex combination of the previous ones. As a consequence, $\mathbb P(R) = \mathbb P_{\reg}(R)$.
\end{prop}

\begin{proof}
Let $\rk$ be any rank function on $R$, and consider the system 
$$
 \rk = \sum_{k=1}^{n} c_k \rk_k, ~c_k\geq 0, ~\sum_{k=1}^{n} c_k = 1.
$$
As observed previously, any rank function on $R$ is determined by its values on $t^i+(t^n)$, $1\le i \le n-1$. Therefore, for this system to have a solution it is enough to find non-negative $c_k$ satysfying 
$$
\rk(t^i+(t^n)) = \sum_{k=1}^{n} c_k \rk_k(t^i+(t^n)) = \sum_{k=i}^{n} c_k \frac{k-i}{k}
$$
for any $0\le i \le n-1$, since the equality for $i=0$ already encodes that the coefficients add up to 1. Setting $b_k = \rk(t^k + (t^n))-\rk(t^{k+1}+(t^n))$ for $0\le k\le n-1$, we obtain that $b_k = \sum_{j = k + 1}^n \frac{1}{j}c_j$, and so the only possible solution is given by
$$
\begin{matrix*}[c]
 c_k = k (b_{k-1} - b_k), ~k = 1,\dots, n-1, & c_n = nb_{n-1}
 \end{matrix*}
$$
Finally, every $c_k\geq 0$ by \cref{lemm:positive_coefficients} and $\sum_{k=1}^n c_k = \sum_{k=0}^{n-1} b_k = \rk(1+(t^n)) = 1$. Since the solution is unique for every rank, $\rk_1,\dots,\rk_n$ are the only rank functions that cannot be expressed as a convex combination of two different ranks and hence they are the extreme points in $\mathbb P(R)$. 

The last assertion follows from the regularity of $\rk_k$ and the convexity of $\mathbb P_{\reg}(R)$.
\end{proof}

While the point of view of matrices is usually easier to understand at first, working with Sylvester module rank functions allows us to use properties that do not have an analog for matrices. In this sense, observe that the associated extreme Sylvester module rank functions $\dim_1,\dots,\dim_n$ are determined by
$$
\dim_k(R/(t^i+(t^n))) = \begin{cases}
\frac{i}{k} & \textrm{if}~ i\leq k \\
1 & \textrm{otherwise}
\end{cases}
$$

\subsection{The general case}
Let us now consider any left artinian local ring $R$. This implies in particular that $J(R)$ is nilpotent (cf. \cite{Pierce1982}*{Proposition~4.4} or \cite{GW2004}*{Theorem~4.15}). 
We show that if we further assume that $J(R)$ is generated by a central element, then essentially the same classification presented above still holds.

Thus, assume that $c\in Z(R)$ is such that $J(R) = (c)$, and let $n$ be the smallest positive integer such that $c^n = 0$. Mirroring the previous example we are going to show that every rank function on $R$ is determined by its values on $c,\dots, c^{n-1}$ and that the expressions
$$
\dim_k(R/(c^i)) = \begin{cases}
\frac{i}{k} & \textrm{if}~ i\leq k \\
1 & \textrm{otherwise}
\end{cases}
$$
can be uniquely extended to Sylvester module rank functions $\dim_1,\dots,\dim_n$ on $R$ that turn out to be the extreme points in $\mathbb P(R)$. 

Let us first study the structure of modules over this local ring. Since $R$ is local, an element $x$ is either a unit or belongs to $J(R)$. Let $x$ be a non-zero element of $J(R)$ and $m$ the positive integer such that $x\in J(R)^m\backslash J(R)^{m+1}$. Since $c\in Z(R)$, then $J(R)^m=(c^m)$ and therefore $x = c^m u$ for some unit $u\in R$.

Now, take $A \in \Mat_{k\times l}(R)$ and express every element of $A$ as $a_{ij} = c^{m_{ij}}u_{ij}$ where $u_{ij}$ is a unit (here, $m_{ij}=0$ if $a_{ij}$ is already a unit, and $m_{ij} = n$ if $a_{ij}=0$). Multiplying by invertible matrices we can assume that $m_{11} = \min\{m_{ij}\}$ and $a_{11}= c^{m_{11}}$. Thus, if $r_{ij} = m_{ij}-m_{11}$, \medskip
$$
\begin{pmatrix} 
 1 & 0 &\dots & 0\\
 -c^{r_{21}}u_{21}& 1 &      &0 \\
  \vdots & & \ddots & \vdots\\
 -c^{r_{k1}}u_{k1} & 0 &\dots & 1
 \end{pmatrix}
  A
  \begin{pmatrix} 
 1 & -c^{r_{12}}u_{12} &\dots & -c^{r_{1l}}u_{1l}\\
 0 & 1 &      &0 \\
  \vdots & & \ddots & \vdots\\
 0 & 0 &\dots & 1
 \end{pmatrix}
 =
 \begin{pmatrix} 
 c^{m_{11}} & 0 \\
 0 & B 
\end{pmatrix}
$$

\medskip
\noindent for some $(k-1)\times (l-1)$ matrix $B$. Using induction we see that $A$ is equivalent to a matrix of the form $\begin{pmatrix} D & 0 
 \end{pmatrix}$ or $\begin{pmatrix} D \\ 0 
 \end{pmatrix}$ where $D$ is a diagonal matrix whose entries are either powers $c^i$ for some $0\le i\le n-1$ or zero. 
 
This expression for a matrix implies that every finitely presented left $R$-module $M$ can be written in the form 
$$M \cong R^m \oplus \bigoplus_{i=1}^r R/(c^{m_i}).$$ 

Moreover, since $c\in Z(R)$ and $R$ is local, for each $0\le i\le n$ we have that $R/(c^i)$ is an $R$-bimodule (with the usual operations) of length $i$. Indeed, under the previous hypothesis one can show that $J(R)^{k-1}/J(R)^k\cong R/J(R)$ for every $1\le k\le n$, from where it is a simple left $R$-module, and hence the chain $$
  0 < J(R)^{i-1}/J(R)^i < \dots < J(R)/J(R)^i < R/J(R)^i
$$
is a composition series for $R/J(R)^i = R/(c^i)$ of length $i$ (in particular, $l({_R}R)=n$).

In addition, we have $R$-bimodule isomorphisms (cf. \cite{Rotman2009}*{Proposition~2.68})
$$
    R/(c^i) \otimes_R R/(c^j) \cong R/(c^{\min\{i,j\}}). 
$$
Using this fact, one can show that the previous expression for the $R$-module $M$ is unique (up to reorganization of factors) by tensoring with $R/(c^i)$ for every $i=1,\dots, n$ and comparing lengths.  

\begin{prop} \label{prop:local_art}
 Let $R$ be a left artinian local ring, and assume that there exists an element $c\in Z(R)$ with order of nilpotency $n$ such that $J(R) = (c)$. Then any rank function on $R$ is determined by its values on $c^i$ for $1\le i \le n-1$, and the expressions $\dim_1,\dots, \dim_n$ extend uniquely to Sylvester module rank functions on $R$ that are precisely the extreme points in $\mathbb P(R)$. Any other rank function can be uniquely expressed as a convex combination of the previous ones.
\end{prop}

\begin{proof}

The previous expression for a finitely presented left $R$-module, together with (SMod2), shows that there exists only one way to extend $\dim_k$ for an arbitrary finitely presented module. More precisely, if we split the decomposition of $M$ as
$$
M = R^m \oplus \bigoplus_{\substack{m_i>k\\i=1,\dots, r_1}} R/(c^{m_i})  \oplus  \bigoplus_{\substack{n_j\leq k \\j=1,\dots, r_2}} R/ (c^{n_j}) 
$$
then $\dim_k(M) = m+r_1+\sum_j \frac{n_j}{k}$. Thus, observe that 
$$
 \dim_k(M) = \frac{l(R/(c^k)\otimes_{R} M)}{k}
$$
where $l(N)$ stands for the length of $N$. Since the tensor commutes with direct sums and it is right exact, and the length of a module is additive on short exact sequences, we deduce that $\dim_k$ satisfies (SMod1)-(SMod3). 

Since any finitely presented module can be written in the above form, we deduce that any Sylvester matrix rank function is determined by its values on $c,\dots,c^{n-1}$. Thus, noting that the associated matrix rank functions are the analogues to the extreme ranks on $K[t]/(t^n)$, the same argument of \cref{prop:Kt_tn} shows that these are the extreme points in $\mathbb P(R)$.
\end{proof}

\begin{coro} \label{coro:primary_art}
  Let $R$ be a left artinian primary ring, and assume that there exists an element $c\in Z(R)$ with order of nilpotency $n$ such that $J(R) = (c)$. Then any rank function on $R$ is determined by its values on $c^i$ for $1\le i \le n-1$, and the extreme points in $\mathbb P(R)$ are the Sylvester matrix rank functions $\rk_1,\dots, \rk_n$ defined by 
  $$
  \rk_k(c^i) = \begin{cases}
\frac{k-i}{k} & \textrm{if}~ i\leq k \\
0 & \textrm{otherwise}
\end{cases}
  $$
  Any other rank function can be uniquely expressed as a convex combination of $\rk_1,\dots, \rk_n$.
\end{coro}

\begin{proof}
  By \cref{prop:Pierce}, there exist a left artinian local ring $S$, a positive integer $s$ and a ring isomorphism $\varphi: R \rightarrow \Mat_s(S)$. Notice that, since $\varphi$ is an isomorphism, we have that  $\varphi(J(R)) = J(\Mat_s(S)) = \Mat_s(J(S))$ and it is generated by $\varphi(c)$.
  
  Now, $c$ is a central element, and hence  $\varphi(c) \in Z(\Mat_s(S))$, from where necessarily $\varphi(c) = \Diag_s(d,\dots,d)$ for some $d\in J(S)\cap Z(S)$ of order $n$. In addition, observe that $J(S) = (d)$.
  
  Thus, in terms of Sylvester matrix rank functions, \cref{prop:local_art} tells us that the extreme points on $\mathbb P(S)$ are the ranks $\rk'_1, \dots, \rk'_n$ given by 
  $$
   \rk'_k(d^i) = \begin{cases}
\frac{k-i}{k} & \textrm{if}~ i\leq k \\
0 & \textrm{otherwise}
\end{cases}
  $$
  and, by \cref{prop:Morita}, $\frac{1}{s}\rk'_1,\dots, \frac{1}{s}\rk'_n$ are the extreme points in $\mathbb P(\Mat_s(S))$. Therefore, since ring isomorphisms preserve the extreme rank functions, we obtain that $\rk_1 = \varphi^{\sharp}(\frac{1}{s}\rk'_1), \dots, \rk_n = \varphi^{\sharp}(\frac{1}{s}\rk'_n)$ are the extreme points in $\mathbb P(R)$, and taking into account that
  $$
  \rk_k(c^i) = \frac{1}{s}\rk'_k(\Diag_s(d^i, \dots, d^i)) = \rk'_k(d^i),
  $$
  $\rk_k(c^i)$ is defined as in the statement. 
  
  To finish, we need to check that rank functions on $R$ are determined by their values on $c,\dots, c^{n-1}$. Given the bijectivity of the maps $\mathbb P(R)\to \mathbb P(\Mat_s(S)) \to \mathbb P(S)$, if two rank functions on $R$ coincide on powers of $c$, their images are rank functions on $S$ that coincide on powers of $d$, and hence they are equal by \cref{prop:local_art}. This finishes the proof.
\end{proof}

\section{Sylvester rank functions on a Dedekind domain} \label{sect:Dedekind}

This section is devoted to describing the space of rank functions $\mathbb P(\OC)$ defined on a Dedekind domain $\OC$ which is not a field, since the latter case has already been treated through \cref{sect:Sylvester}.

Recall that over a Dedekind domain every non-zero prime ideal is maximal
and every non-zero proper ideal can be represented uniquely as a finite product of powers of distinct prime ideals (cf. \cite{BK2000}*{Lemma~5.1.18 \& Theorem~5.1.19}). Moreover, every finitely generated $\OC$-module $M$ can be expressed as follows
$$
  M\cong \OC^n \oplus I \oplus \left(\bigoplus_{j=1}^m \OC/\mf_j^{\alpha_j}\right),
$$
where $I$ is an ideal of $\OC$ and the $\mf_j$ are (non-necessarily distinct) maximal ideals (cf. \cite{BK2000}*{Theorem~6.3.23}, where the uniqueness of such decomposition is also discussed).

The following lemma collects some other basic properties of ideals over Dedekind domains that we shall need for our purposes.

\begin{lemm}\label{lemm:Dedekind_properties}
Let $I$ be any non-zero ideal over the Dedekind domain $\OC$. Then, the following hold.
\begin{enumerate} 
 \item The quotient $\OC/I$ is an artinian principal ideal ring. 
 
 \item $I$ is projective and strongly two-generated, i.e., for every non-zero $x\in I$, there exists $y\in I$ such that $I = \OC x+ \OC y$.
 
 \item For every non-zero ideal $J$, there exists an $\OC$-isomorphism $I \oplus J \cong \OC \oplus IJ$.
\end{enumerate}
\end{lemm}

\begin{proof}
The statements correspond, respectively, to \cite{BK2000}*{Proposition~5.1.22}, \cite{BK2000}*{Corollary~5.1.23 \& Lemma~6.1.1} and \cite{BK2000}*{Lemma~6.1.4}.
\end{proof}

Observe from \cref{lemm:Dedekind_properties}(1) that for every maximal ideal $\mf$ and positive integer $n$, the quotient ring $\OC/\mf^n$ is a local artinian ring whose unique maximal ideal $J(\OC/\mf^n) = \mf/\mf^n$ is (nilpotent and) principal. Let $c\in \mf$ be such that $c+\mf^n$ generates $\mf/\mf^n$ and note that, since different powers of a proper ideal in $\OC$ are all distinct (for instance, because of the uniqueness of a primary decomposition; see also \cite{BK2000}*{Proposition~5.1.24}), $n$ is precisely the order of nilpotency of this element. Then, \cref{prop:local_art} tells us that there are exactly $n$ extreme Sylvester matrix rank functions $\rk_{\OC/\mf^n,1},\dots,\rk_{\OC/\mf^n,n}$ in $\mathbb P(\OC/\mf^n)$, which are determined by
  $$
  \rk_{\OC/\mf^n,k}(c^i+\mf^n) = \begin{cases}
\frac{k-i}{k} & \textrm{if}~ i\leq k \\
0 & \textrm{otherwise}
\end{cases}
  $$
Hence, we can define Sylvester matrix rank functions on $\OC$ through the ring homomorphisms $\pi_{\mf,n}: \OC\to \OC/\mf^n$. In particular, for any maximal ideal $\mf$ and positive integer $k$ we have a rank function $\rk_{\mf, k} = \pi_{\mf, k}^{\sharp}(\rk_{\OC/\mf^k,k})$. We are going to show that the associated Sylvester module rank functions $\dim_{\mf,k}$ satisfy, for every maximal ideal $\nf$ and positive integer $i$,
$$
\dim_{\mf,k}(\OC/\nf^i) = \begin{cases}
\frac{i}{k} & \textrm{if}~ \nf = \mf ~\textrm{and}~ i\leq k \\
1 & \textrm{if}~ \nf = \mf ~\textrm{and}~ i > k \\
0 & \textrm{if}~ \nf \neq \mf
\end{cases}
$$
Assume first that $\nf \neq \mf$, take any non-zero $x\in \nf^i\backslash \mf$ and let $y$ be as in \cref{lemm:Dedekind_properties}(2), so that $\nf^i = \OC x + \OC y$ . Then $\begin{pmatrix} x\\ y\end{pmatrix}$ is a presentation matrix for $O/\nf^i$, and therefore
$$
  \dim_{m,k}(\OC/\nf^i) = 1-\rk_{\OC/\mf^k,k}\begin{pmatrix} x+\mf^k\\ y+\mf^k\end{pmatrix} = 0,
$$
where the last equality follows because, since $x\notin \mf$, $x+\mf^k$ is invertible in $\OC/\mf^k$.

Now, for every $i< k$, let $x$ be a non-zero element of $\mf^k$, take $y$ such that $\mf^i = \OC x + \OC y$ and observe that necessarily $y\in \mf^i\backslash \mf^{i+1}$. Thus, if $c\in \mf$ is such that $c+\mf^k$ generates $\mf/\mf^k$, there exists an element $r\in \OC\backslash \mf$, hence a unit in $\OC/\mf^k$, such that $y+\mf^k = (r+\mf^k)(c^i+\mf^k)$, and as before,
$$
 \dim_{m,k}(\OC/\mf^i) = 1-\rk_{\OC/\mf^k,k}\begin{pmatrix} x+\mf^k\\ y+\mf^k\end{pmatrix} = 1-\rk_{\OC/\mf^k,k}\begin{pmatrix} 0\\ c^i+\mf^k\end{pmatrix} = \frac{i}{k}.
$$
Finally, if $i\ge k$, the generators of $\mf^i$ are zero in $\OC/\mf^k$, and $\dim_{m,k}(\OC/\mf^i) = 1$.\smallskip

As a remark here, observe that the local artinian ring $(R = \OC/\mf^n, J = \mf/\mf^n)$ is complete (i.e. the natural map $R\to \lim_i R/J^i$ is an isomorphism) and separated with respect to the $J$-adic topology (i.e., $\bigcap_{i} J^i = 0$), since $J$ is nilpotent. Thus, if we further assume that $\OC$ contains a field $k$, then $k\hookrightarrow \OC \to R$ is injective, and hence Cohen's theorem on local rings (cf. \cite{Matsumura1980}*{Theorem~60 \& the proof of the subsequent Corollary~1}) tells us that there exists a subfield $K$ of $R$, with $K \cong R/J \cong \OC/\mf$, and a surjective ring homomorphism $K[[t]]\to R$ where $t$ maps to the generator of $J$. In particular, $t^n$ is the first power of $t$ in the kernel of the map, and since $K[[t]]$ is a discrete valuation ring, this shows that the kernel must be precisely $(t^n)$. Hence, we have an isomorphism $K[t]/(t^n)\cong R$, meaning in particular that in this case all rank functions on $\OC/\mf^n$ are regular by \cref{prop:Kt_tn}.

We can also define the regular rank function $\rk_0$ on $\OC$ induced by the field of fractions $\mathcal Q(\OC)$. The same arguments show that the associated module rank function $\dim_0$ satisfies $\dim_0(\OC/\mathfrak \nf^i) = 0$ for every maximal ideal $\nf$ and positive integer $i$.

The structure of finitely generated modules over $\OC$ and the next proposition show that, in general, any Sylvester module rank function $\dim$ on $\OC$ is determined by its values on the modules considered above. 

\begin{prop}\label{prop:rank_of_ideals}
If $\dim$ is a Sylvester module rank function over $\OC$, then $\dim(I) = 1$ for every non-zero ideal $I$ of $\mathcal O$.
\end{prop}

\begin{proof}
Let $I$ be a non-zero ideal of $\mathcal O$. By using  \cref{lemm:Dedekind_properties}(3) repeatedly, we can see that for every positive integer $k$ we have $\bigoplus_{i=1}^k I \cong \OC^{k-1}\oplus I^k$, from where 
$$
\dim(I) = \frac{k-1}{k} + \frac{\dim(I^k)}{k} \ge \frac{k-1}{k}.
$$
Thus, necessarily $\dim(I)\ge 1$. On the other hand, $I$ is projective and two-generated by \cref{lemm:Dedekind_properties}(2), and hence a direct summand of $\mathcal \OC^2$. Since an ideal of $\OC$ cannot be free of rank $2$, its complement $C$ must be non-zero, and the structure of finitely generated modules over $\OC$ together with the previous argument shows that $\dim(C)\ge 1$. Since $\dim(I)+\dim(C) = \dim(\OC^2) = 2$, necessarily $\dim(I)= 1$.  
\end{proof}

We are going to show that the previous ranks $\dim_{\mf,k}$ and $\dim_0$ are the extreme points in $\mathbb{P}(\mathcal O)$ by proving that any rank function $\dim$ on $\mathcal O$ can be uniquely written as
$$
 \dim = c_0\dim_0 + \sum_{\mf}\sum_{k\in \ZZ^+} c_{\mf,k} \dim_{\mf,k}, ~~~c_0, c_{\mf,k}\geq 0, ~c_0+\sum c_{\mf,k} = 1,
$$
where $\mf$ runs through all maximal ideals of $\OC$ and $k$ through the positive integers. As there can be an uncountable number of maximal ideals in $\OC$, for the right hand side sum to make sense we need to show first that there are only countably many non-zero coefficients. For this purpose, note that if such an expression is to hold, then from the definition of $\dim_{\mf, k}$ we obtain that
$$
\dim(\OC/\mf) = \sum_{k\ge 1} \frac{c_{\mf,k}}{k}.
$$
In particular, if $\dim(\OC/\mf) = 0$, then $c_{\mf, k} = 0$ for every $k$. Thus, for our goal it suffices to see that $\dim(\OC/\mf) = 0$ for all but countably many maximal ideals.

Notice also that the previous equality implies that if $\dim(\OC/\mf) = 0$, then $\dim(\OC/\mf^k) = 0$ for every $k$, a statement that will follow from \cref{lemm:positive_coefficients_Dedekind} in our case but may not be true in general for a commutative ring $R$ (From (SMod3) and the surjective homomorphism $R/\mf^n\to R/\mf$ we only deduce $\dim(R/\mf^n) \ge \dim(R/\mf)$). However, we show in the following lemma that the number of pairs $(\mf, k)$ such that $\dim(R/\mf^k) > 0$ is still countable in this more general setting.

\begin{lemm} \label{lemm:countably_many_maximals_comm}
  Let $\dim$ be a Sylvester module rank function on a commutative ring $R$. Then, there exist only countably many maximal ideals $\mathfrak{m}$ such that $\dim(R/\mathfrak{m}^k)>0$ for some $k\ge 1$.
\end{lemm}

\begin{proof}
Fix $k\ge 1$ and, for every $n\ge 1$, let $S_n^{(k)}$ be the collection of all maximal ideals of $R$ with $\dim(R/\mf^k)>1/n$. Suppose that there exists $n$ such that $S_n^{(k)}$ is infinite, and take $m>n$ different maximal ideals $\{\mf_i\}_{i=1}^m$ on $S_n^{(k)}$. Then, using the Chinese Remainder Theorem and (SMod2),
$$
\dim(R/\mathfrak{m}_1^k\cdots\mathfrak{m}_m^k) = \sum_{i = 1}^m \dim (R/\mathfrak{m}_i^k) > \frac{m}{n} > 1.
$$ 
This is a contradiction, since for every ideal $I$ in $R$, we have $\dim(R/I)\leq \dim(R)=1$ by (SMod1) and (SMod3), so $S_n^{(k)}$ must be finite for every $n$. Therefore, the set $S_0 = \bigcup_{k\in \ZZ^+} \bigcup_{n\in \mathbb{Z}^+} S_n^{(k)}$ is countable.
\end{proof}

Since the computation of coefficients is going to be very similar to that of Proposition \ref{prop:Kt_tn}, we need also the following generalization of \cref{lemm:positive_coefficients}.

\begin{lemm} \label{lemm:positive_coefficients_Dedekind}
Let $\dim$ be a Sylvester module rank function on $\OC$ and let $\mf$ be a maximal ideal. If we set $b_{\mf, 0} = \dim(\OC/\mf)$ and $b_{\mf,k} = \dim(\OC/\mf^{k+1})-\dim(\OC/\mf^{k})$ for $k\ge 1$, then $b_{\mf,k} \geq b_{\mf,k+1}$ for every $k\ge 0$.
\end{lemm}

\begin{proof}
For any non-zero $x\in \mf^2$, we can find a non-zero element $y\in \mf\backslash\mf^2$ such that $\mf = \OC x + \OC y$ by \cref{lemm:Dedekind_properties}(2). Observe that we can write then $\mf^{k+1} = \mf x^{k} + \mf^{k-1} y^2$ for every $k\ge 1$ (with $\mf^0 = \OC$). One can check that the sequences
$$
\OC/\mf \xrightarrow{\varphi_0} \OC/\mf^2 \xrightarrow{\psi_0} \OC/\mf \to 0,
$$
where $\varphi_0(r+\mf) = yr+\mf^2$ and $\psi_0(s+\mf^2)=s+\mf$, and 
$$
\OC/\mf^{k+1} \xrightarrow{\varphi_k} \OC/\mf^{k+2} \oplus  \OC/\mf^k \xrightarrow{\psi_k} \OC/\mf^{k+1} \to 0,  
$$
where $\varphi_k(r+\mf^{k+1})= (yr+\mf^{k+2}, r+\mf^k)$ and $\psi_k(s+\mf^{k+2}, t+\mf^{k})= s-yt + \mf^{k+1}$, are all exact. In fact, every $\varphi_k$ for $k\ge 0$ can be shown to be injective, but this is not needed for the proof since the result already follows from (SMod2) and (SMod3).
\end{proof}
We are now ready to prove the main result of this section about $\mathbb P(\mathcal O)$.

\begin{theo} \label{theo:Dedekind}
The Sylvester module rank functions $\dim_0$ and $\dim_{\mf, k}$, for every maximal ideal $\mf$ and $k\ge 1$, are the extreme points of $\mathbb{P}(\OC)$, and any other rank function can be uniquely expressed as a (possibly infinite) convex combination of them. In particular, if $\OC$ contains a field, $\mathbb{P}(\OC) = \mathbb{P}_{\reg}(\OC)$.
\end{theo}

\begin{proof}
Let $\dim$ be any Sylvester module rank function on $\OC$. By  \cref{lemm:countably_many_maximals_comm}, the set $S_0$ of all maximal ideals $\mf$ such that $\dim(\OC/\mf^k)>0$ for some $k\ge 1$ is countable, and as we already discussed, only coefficients $c_{\mf, k}$ corresponding to $\mf\in S_0$ can be non-zero. Thus, we are going to show that there are unique non-negative numbers $c_0, c_{\mf,k}$ summing up to 1 such that
$$
\dim = c_0\dim_0 + \sum_{\mf\in S_0} \sum_{k\in \mathbb{Z^+}} c_{\mf,k} \dim_{\mf,k}.
$$
As we deduced from the structure of finitely generated modules over $\OC$ and \cref{prop:rank_of_ideals}, it suffices to have equality for every module $\OC/\nf^i$, where $\nf$ is a maximal ideal and $i\ge 1$. From the definition of $\dim_{\mf, k}$, its only contribution for these modules is given when $\mf = \nf$. In particular, if $\nf\notin S_0$, the two expressions coincide on $\OC/\nf^i$ for every $i$, and if $\nf\in S_0$, the coefficients $c_{\nf,k}$ should be determined by
$$
\dim(\OC/\nf^i) = \sum_{k=1}^{i-1}c_{\nf,k}  + \sum_{k=i}^\infty c_{\nf,k}\frac{i}{k}
$$  
Borrowing the notation of \cref{lemm:positive_coefficients_Dedekind}, we obtain that for every $i\ge 0$,
$$
b_{\nf,i} = \sum_{k = i+1}^\infty \frac{1}{k}c_{\nf,k}.
$$
Thus the only possible solution is given by the non-negative coefficients (\cref{lemm:positive_coefficients_Dedekind}) 
$$
 c_{\nf,k} = k(b_{\nf,k-1}-b_{\nf,k}), k\ge 1.
$$
We still need to show that $\sum_{\mf\in S_0} \sum_{k=1}^\infty c_{\mf,k}$ converges to a number $l$ less than or equal to 1, and take $c_0 = 1-l$. Notice first that $\sum_{k=0}^n b_{\mf,k} = \dim(\mathcal O/\mf^{n+1})$, and hence the sequence of partial sums $\{\sum_{k=0}^n b_{\mf,k}\}_n$ is monotonically increasing and bounded above by $1$, so the series $\sum_{k=0}^{\infty} b_{\mf,k}$ is convergent for every $\mf$. Moreover, since by (SMod3) and \cref{lemm:positive_coefficients_Dedekind} we have $b_{\mf,k}\ge b_{\mf,k+1}\ge 0$ for every $k\ge 0$, Abel-Pringsheim theorem (cf. \cite{Hardy2008}*{\S 179}) tells us that $\lim_{k\to \infty} kb_{\mf,k} = 0$. Therefore, from the inequalities
$$
 0\le \sum_{k=1}^{n} c_{\mf, k} = \sum_{k=1}^{n} k(b_{\mf,k-1}-b_{\mf,k}) = \left[\sum_{k = 0}^{n-1}b_{\mf, k}\right]-nb_{\mf,n} \le \sum_{k = 0}^{n-1}b_{\mf, k} \le 1,
$$
we obtain, on the one hand, that the sequence of partial sums $\{\sum_{k=1}^n c_{\mf,k}\}_n$ is also monotonically increasing and bounded above by $1$, and hence that the series $\sum_{k=1}^{\infty} c_{\mf, k}$ is also convergent. On the other hand, we also deduce from the previous discussion that $\sum_{k=1}^{\infty} c_{\mf, k} = \sum_{k=0}^{\infty} b_{\mf, k}$. We claim that $\sum_{\mf\in S_0} \sum_{k=0}^\infty b_{\mf,k}$ converges to a number smaller than 1, from where the result is established. Indeed,
$$
\sum_{\mf\in{S_0}}\sum_{k=0}^\infty b_{\mf,k} = \sup_{\substack{B\subset S_0 \\B~ \textrm{finite}}} \sum_{\mf\in{B}}\sum_{k=0}^\infty b_{\mf,k},
$$ 
and hence, it suffices to prove that for every finite $B\subset S_0$, $\sum_{\mf\in{B}}\sum_{k} b_{\mf,k}$ converges to a number below $1$. But this follows from the Chinese Remainder Theorem, (SMod2) and (SMod3), since
$$
\sum_{\mf\in{B}} \sum_{k=0}^n b_{\mf,k} = \sum_{\mf\in{B}} \dim(\mathcal O/\mf^{n+1}) = \dim(\mathcal O/\underset{\mf\in B}{\cap} \mf^{n+1}) \leq 1
$$
and therefore
$$
 \sum_{\mf\in B} \sum_{k=0}^{\infty} b_{\mf,k} = \sum_{\mf\in B}\left[\lim_{n\to \infty} \sum_{k=0}^{n} b_{\mf,k}\right] = \lim_{n\to \infty} \left[\sum_{\mf\in B}\sum_{k=0}^{n} b_{\mf,k} \right] \le 1,
$$
where the second equality follows because we are adding a finite number of finite limits.
The last assertion of the proposition is a consequence of the previous discussion regarding $\OC/\mf^n$ and $K[t]/(t^n)$, and that $\mathbb P_{\reg}(\OC)$ is closed and convex.
\end{proof}
\begin{rem*}
It may be needed to clarify why the right hand side expression in the previous theorem actually defines a Sylvester module rank function on $\OC$. 

In general, the identification between the spaces of Sylvester module and matrix rank functions given by the rules in \cref{prop:Malcolmson} is actually a convex-linear homeomorphism between them, reason why we use indistinctly $\PP(R)$ to denote both. Now, the space $\RR^{\Mat(R)}$ of real-valued functions on matrices over a ring $R$ is a locally convex, Haussdorf, topological vector space, and since $\PP(R)$ is a compact convex subset, countably infinite convex combinations of Sylvester rank functions are allowed and, moreover, convex-linear continuous maps between compact convex sets preserve these infinite convex combinations (cf. \cite{Goodearl1991}*{Proposition
~A.7} and the subsequent discussion).
\end{rem*}

\section{Krull dimension and simple left noetherian rings} \label{sect:simple_noetherian}

In this section we turn to the study of simple left noetherian rings. These rings appear naturally when dealing with skew Laurent polynomial rings since, for instance, for every automorphism of infinite inner order $\tau$ (see \cref{sect:Laurent}) of a division ring $\DC$, the ring $\DC[t^{\pm 1}; \tau]$ is simple and noetherian (\cite{GW2004}*{Corollary~1.15 \& Theorem~1.17}). Another widely studied subfamily here are the Weyl algebras $A_n(K)$ over a field of characteristic zero (\cite{GW2004}*{Exercise~2G \& Corollary~2.2}).

We show that on a simple left noetherian ring there exists a unique Sylvester rank function, namely, the one induced from its classical left quotient ring (i.e. the left ring of fractions with respect to the set of all non-zero divisors). This is proved by means of induction on Krull dimension of modules, and we follow the exposition in \cite{GW2004}*{Chapters~15 \& 16} to recall the necessary definitions and results. 

Let $R$ be a ring and let $M$ be a left $R$-module. We say that the \emph{Krull dimension} of $M$ is $-1$, and we write $\kdim(M) = -1$, if and only if $M$ is the zero module. Now, given an ordinal $\alpha\ge 0$, we write $\kdim(M)\leq \alpha$ if, for every descending chain $M_1\ge M_2\ge M_3\dots$ of submodules of $M$, we have $\kdim(M_i/M_{i+1})<\alpha$ for all but finitely many $i$. The Krull dimension of a non-zero module $M$ is then $\alpha$, denoted $\kdim(M) = \alpha$, if $\kdim(M)\leq \alpha$ and $\alpha$ is the least such ordinal, and we write $\lkdim(R)$ to denote the Krull dimension of $R$ as a left $R$-module.

In addition, if the module $M$ has $\kdim(M)=\alpha\ge 0$ and all its proper factor modules have Krull dimension $<\alpha$, i.e., $\kdim(M/N)<\alpha$ for every non-zero submodule $N$ of $M$, then $M$ is called \emph{$\alpha$-critical}.

Observe, for example, that a non-zero module $M$ has Krull dimension 0 if and only if it is artinian. Notice also that for every division ring $\DC$ and every automorphism $\tau$ of $\DC$, the skew Laurent polynomial ring $R = \DC[t^{\pm 1};\tau]$ is not left artinian, since we have the infinite descending chain of left ideals
$$
 R \supseteq R(1+t) \supseteq R(1+t)^2 \supseteq \dots,
$$
and its Krull dimension is at most $1$ (\cite{GW2004}*{Theorem~15.19 \& Exercise~15S}). Thus, $\lkdim(\DC[t^{\pm 1};\tau]) = 1$. 

In general, an ordinal $\alpha$ as in the definition may not exist, so there are modules for which the Krull dimension is not defined. However,  this is not the case of noetherian modules over a ring and, in particular, of finitely generated left modules over a left noetherian ring (\cite{GW2004}*{Lemma~15.3}).

The key point to prove the main result is the following lemma due to Stafford (cf. \cite{Stafford1976}*{Lemma~1.4}). He originally considered modules $M$ with finite Krull dimension, since it turns out to give an upper bound for the minimal number of generators of $M$, but 
the same result holds without this assumption. We add a proof here for the sake of completeness, just following the lines of \cite{GW2004}*{Theorem~16.7}. For this purpose, recall that an $R$-module $M$ is \emph{faithful} if $\ann_R(M) = 0$, \emph{fully faithful} if all its non-zero submodules are faithful, and \emph{completely faithful} if all its non-zero factor modules are fully faithful.   

\begin{lemm}\label{lemm:Stafford}
 Let $R$ be a left noetherian ring and $M$ a non-zero finitely generated completely faithful left $R$-module. If $\kdim(M) < \lkdim(R)$, then there exists a cyclic submodule $N$ of $M$ such that $\kdim(M/N)<\kdim(M)$. 
\end{lemm}

\begin{proof}
  Assume $\kdim(M) = \alpha\ge 0$, and let $J_{\alpha}(M)$ denote the intersection of the kernels of all homomorphisms from $M$ to $\alpha$-critical modules (i.e., the \emph{Krull radical of $M$}). Then, $J_{\alpha}(M)$ is a proper submodule of $M$, the factor module $M/J_{\alpha}(M)$ has Krull dimension $\alpha$ (\cite{GW2004}*{Proposition~15.11}) and hence it is fully faithful by hypothesis. 
  Thus, \cite{GW2004}*{Lemma~16.4} tells us that there exists $m\in M$ such that $(Rm+J_{\alpha}(M))/J_{\alpha}(M)$ is an essential submodule of $M/J_{\alpha}(M)$, and this is the case if and only if $\kdim(M/Rm)<\alpha$ by \cite{GW2004}*{Corollary~15.12}.
\end{proof}

With this, we can now state the main result.

\begin{prop}\label{prop:simple_noet}
  If $R$ is a left noetherian simple ring, then $\mathbb P(R) = \{\dim_{\mathcal Q_l(R)}\}$, where $\mathcal Q_l(R)$ is the classical left quotient ring of $R$. 
\end{prop}

\begin{proof}
 Observe that since $R$ is left noetherian and simple, the classical left quotient ring of $R$ exists and it is simple artinian (cf. \cite{GW2004}*{Corollary~6.19}). Therefore, $\mathcal Q_l(R)$ is (isomorphic to) a matrix ring over a division ring and hence it has only one rank function, that we denote by $\dim_{\mathcal Q_l(R)}$.
 
 Notice also that the simplicity of $R$ implies that every non-zero left $R$-module is faithful. In particular, every non-zero finitely generated left $R$-module is completely faithful. We are going to use \cref{lemm:Stafford} and transfinite induction to show that for every Sylvester module rank function $\dim$ on $R$ and for every finitely generated $M$ with $\kdim(M)<\lkdim(R)$ (equivalently, for every finitely generated torsion module, see \cite{GW2004}*{Proposition~15.7}), we have $\dim(M) = 0$.
 
 The case $\kdim(M)=-1$ follows from (SMod1) and, at every inductive step, if $M$ is finitely generated with $\kdim(M)<\lkdim(R)$, then so is $M^k$, the direct sum of $k$ copies of $M$, for every positive integer $k$ (cf. \cite{GW2004}*{Corollary~15.2}). Therefore, \cref{lemm:Stafford} tells us that $M^k$ contains a cyclic submodule $N$ such that $\kdim(M^k/N)<\kdim(M^k)$. By induction hypothesis, $\dim(M^k/N) = 0$ and, since $N$ is cyclic, (SMod3) tells us that $\dim(N)\le 1$. 
 Thus, from the short exact sequence $0\to N\to M^k\to M^k/N\to 0$, we obtain that $\dim(M)\le \frac{1}{k}$ for every $k$ in view of (SMod2) and (SMod3). Therefore, the previous claim follows.
 
 Taking into account that, for every non-zero-divisor $x\in R$, right multiplication by $x$ defines an injective endomorphism of left $R$-modules $R\to R$, we obtain $\kdim(R/Rx)<\lkdim(R)$ by \cite{GW2004}*{Lemma~15.6}. Thus, if $\rk$ denotes the Sylvester matrix rank function associated to $\dim$, we deduce
 $$
  \rk(x) = 1-\dim(R/Rx)= 1
 $$
 Therefore, by (\ref{Ore}), $\rk$ can be extended to $\mathcal Q_l(R)$, and by uniqueness of the rank in $\mathcal Q_l(R)$, $\dim$ must be the rank induced by $\dim_{\mathcal Q_l(R)}$.
\end{proof}

We finish the section with another consequence of \cref{lemm:Stafford} that will prove useful later. We say that a ring $R$ is \emph{almost simple} (as introduced in \cite{Jaikin1999}) if every non-zero two-sided ideal of $R$ contains a non-zero element from its center $Z(R)$. We shall say that a left $R$-module $M$ is \emph{$Z(R)$-torsionfree} if for all non-zero $c\in Z(R)$ and for all non-zero $m\in M$, we have $cm \ne 0$ (Note that in the language of \cite{GW2004}*{page~81} this would be called a $Z(R)\backslash \{0\}$-torsionfree module). 

\begin{prop} \label{prop:almost_simple}
Let $R$ be a left noetherian almost simple ring with center $Z(R)$, and let $\dim$ be a Sylvester module rank function on $R$. If $\lkdim(R)\ge 1$ and $M$ is a $Z(R)$-torsionfree left $R$-module of finite length, then $\dim(M)=0$. 
\end{prop}

\begin{proof}
 Notice first that every non-zero $Z(R)$-torsionfree module $M$  with finite length is completely faithful. Indeed, assume that there exists $N\lneq M$ such that $M/N$ is not fully faithful, and let $L$ be a submodule of $M$ with $N\lneq L$ and $L/N$ not faithful. Then $\ann_R(L/N)$ is a non-zero two-sided ideal of $R$, and hence it contains a non-zero element $c\in Z(R)$. Thus, $cL\subseteq N \subsetneq L$. 
 However, since $L\le M$, we have that $L$ is $Z(R)$-torsionfree of finite length and therefore the map $\varphi_c: L\to L$ defined by left multiplication by $c$, which is an $R$-homomorphism because $c\in Z(R)$, is injective. By additivity of length (cf. \cite{GW2004}*{Proposition~4.12}), it must also be surjective, from where $L= \im \varphi_c =cL$, a contradiction.
 
 Observe now that for every $k\ge 1$, $M^k$ is also $Z(R)$-torsionfree and has finite length. 
 In particular, $\kdim(M^k)=0$ and we can apply \cref{lemm:Stafford} to deduce that $M^k$ is cyclic. 
 By (SMod3), $\dim(M^k) \le 1$ for every $k$, and hence (SMod2) implies that $\dim(M)\le \frac{1}{k}$ for every $k$, from where $\dim(M) = 0$.
\end{proof}

\section{Skew Laurent polynomials over division rings} \label{sect:Laurent}

This section focuses on describing the space of rank functions associated to a skew Laurent polynomial ring $ \DC[t^{\pm 1};\tau]$, where $\DC$ is a division ring and $\tau$ is an automorphism of $\DC$. Unless otherwise specified, throughout this section $R$ denotes a skew Laurent polynomial ring of this form.

The main result of the section describes the extreme rank functions on $R$ and shows that every rank function on $R$ is the unique extension of a rank function on its center $Z(R)$ (cf. \cref{Question}). We also notice at the end of the section that the same theory developed for this ring can be applied to the usual polynomial ring $\DC[t]$ and, more generally, to polynomial rings with coefficients on a simple artinian ring.

In order to prove this, we need to recall the structure of two-sided ideals and the center of $R$, and we start with the description of the latter in the following lemma (cf. \cite{BK2000}*{Lemma~3.2.14}). Recall that the \emph{inner order} of an automorphism $\tau$ is the smallest positive integer $m$ such that $\tau^m$ is an inner automorphism, and we say that $\tau$ has infinite inner order if no positive power of $\tau$ is inner. Over a division ring $\DC$, infinite inner order of the automorphim $\tau$ of $\DC$ is equivalent to the simplicity of $\DC[t^{\pm 1};\tau]$ (cf. \cite{GW2004}*{Theorem~1.17}).

\begin{lemm} \label{lemm:Laurent_center}
Denote $K = Z(\DC)$, and let $K^{\tau}$ denote the subfield of $K$ formed by the elements of $K$ fixed by the automorphism $\tau$ of $\DC$. Then,
\begin{enumerate}
  \item[(i)] If $\tau$ has infinite inner order, then $Z(R) = K^{\tau}$.  
  \item[(ii)]If $\tau$ has inner order $m$, say $\tau^m(d)=a^{-1}da$ for some $a\in \DC$, and  $k$ is the smallest positive integer for which there exists a non-zero $b\in K$ such that $\tau(ba^k) = ba^k$, then $Z(R) = K^{\tau}[(ba^kt^{km})^{\pm 1}]$.
\end{enumerate}
\end{lemm}

\begin{proof}
  Take an element $p \in Z(R)$, $p = \sum a_i t^i$. Since it must commute with every element $d\in \DC$, we obtain that $da_i = a_i\tau^i(d)$ for every $i$, and since $p$ commutes with $t$, we have $a_i = \tau(a_i)$. Thus, if $\tau$ has infinite inner order, we obtain from the first condition that $a_i = 0$ for every $i\ne 0$ and $a_0\in K$, and from the second that actually $a_0\in K^{\tau}$. Thus, $p\in K^{\tau}$. Since clearly $K^{\tau}$ lies in the center, we obtain $(i)$. 
  
  If, on the contrary, $\tau$ has inner order $m$, then we obtain from the first condition that $a_i = 0$ for every $i$ not dividing $m$, and that $\tau^{mj}$ is given by conjugation by $a_{mj}$ whenever it is non-zero.
  But $\tau^{mj}$ is also given by conjugation by $a^j$, so for every $d$,
    $$
      a^{-j}da^j = a_{mj}^{-1}da_{mj}.
    $$
  We deduce that $a_{mj}a^{-j}\in K$, i.e., there exists a non-zero $c_j\in K$ such that $a_{mj} = c_ja^j$. 
  From the second condition we obtain now that $a_0 = c_0\in K^{\tau}$, and for every $j\ne 0$, $\tau( c_ja^j) = c_ja^j$, from where the choice of $k$ implies that $k\leq |j|$. 
  We claim that $k$ divides $|j|$ and $c_jb^{-j/k} \in K^{\tau}$. Indeed, let $\sg(j)$ denote the sign of $j$ (i.e., $\sg(j)=1$ if $j$ is positive and $-1$ if it is negative), so that $c_j^{\sg(j)}a^{|j|} = (c_ja^{j})^{\sg(j)}$, and notice that if $|j| = kn+l$ for some $n\ge 1$ and $0\le l<k $, then $c_j^{\sg(j)}b^{-n}\in K$ and $$\tau(c_j^{\sg(j)}b^{-n}a^l) = 
  \tau(c_j^{\sg(j)}a^{|j|}(ba^k)^{-n}) = c_j^{\sg(j)}a^{|j|}(ba^k)^{-n} = c_j^{\sg(j)}b^{-n}a^l.$$ 
  The minimality of $k$ implies $l = 0$ and thus $c_j^{\sg(j)}b^{-n}\in K^{\tau}$, so $c_jb^{-j/k} \in K^{\tau}$.
  
  As a consequence, there exists $z_r\in K^{\tau}$ such that $p = \sum_r z_rb^ra^{kr}t^{mkr}$. Since $(ba^kt^{mk})^r = b^r a^{rk} t^{rmk}$ and, for every $d\in \DC$, $d(ba^kt^{mk})^r = (ba^kt^{mk})^r d$, $p$ can be seen as a Laurent polynomial in the commuting variable $ba^kt^{mk}$ with coefficients in $K^\tau$.  Conversely, every such polynomial $p$ belongs to the center. Thus, $Z(R)$ is the ordinary Laurent polynomial ring $K^{\tau}[(ba^kt^{mk})^{\pm 1}]$ and we have proved $(ii)$.
  \end{proof}
  
Note that there always exists an element $b$ as in the hypothesis. Indeed, when $\tau$ has finite inner order $m$ we have that $R$ is not simple and, as we would have also deduced from \cref{lemm:ideals_Laurent}(1.), $Z(R)\cap \DC[t;\tau]\nsubseteq \DC$ (For a concrete example, see \cite{GW2004}*{Exercise~1U}). Now, for every $p\in (Z(R)\cap \DC[t;\tau])\backslash \DC$, we saw in the proof that every non-zero coefficient corresponding to positive degree $mj$ is of the form $c_ja^j$, $c_j\in K$ and $\tau(c_ja^j) = c_ja^j$, as desired.

More than its precise description, it is important to observe that if $R$ is non-simple, the center of $R$ is an ordinary Laurent polynomial ring over a field, and hence every result in this section also applies to the center. When $R$ is simple, the center is a field and we obtain the following from \cref{prop:simple_noet}.

\begin{coro} \label{coro:simple_Laurent}
  If $\tau$ has infinite inner order, the only Sylvester rank function on $R$ is the one coming from its Ore quotient ring, and extends the unique Sylvester rank function on $Z(R)$.
\end{coro}

The last assertion of the corollary follows since in $Z(R)$, a field, there exists a unique rank function, which is then necessarily the restriction of the unique rank function in $R$.

For the rest of the section, assume that $\tau$ has finite inner order and let us denote $S = Z(R)\cap \DC[t;\tau]$, which is an ordinary polynomial ring over a field by \cref{lemm:Laurent_center}. In the following lemma we relate two-sided ideals of $R$ and elements of $S$. Recall that a non-constant $p\in S$ is \emph{irreducible} if it cannot be expressed as a product $p=rq$ for some non-constant $r,q\in S$. Recall also that $p,q\in S$ are said to be \emph{associates} if there exist $r,r'\in S$ such that $p=rq, q=r'p$. A comparison of degrees shows that this is the case if and only if $p=dq$ for some unit $d\in Z(R)\cap \DC$.

\begin{lemm} \label{lemm:ideals_Laurent}
  The following hold:
  \begin{enumerate}
      \item[(1.)] Let $I$ be a non-zero proper two-sided ideal of $R$. Then, there exists a non-constant polynomial $p\in S$ with non-zero constant term such that $p$ generates $I$ and which is irreducible in $S$ if and only if $I$ is maximal.
      \item[(2.)] There exists a bijective correspondence between maximal two-sided ideals in $R$ and irreducibles in $S$ with non-zero constant term up to association. 
      This defines a bijective correspondence between maximal two-sided ideals $\mf$ of $R$ and maximal ideals of $Z(R)$ sending $\mf$ to $\mf_Z = \mf \cap Z(R)$.
      \item[(3.)] If $\mf_1, \dots, \mf_n$ are different maximal two-sided ideals in $R$, then for all positive integers $k_1,\dots,k_n$, we have $\bigcap_i \mf_i^{k_i} = \mf_1^{k_1}\dots \mf_n^{k_n}$.
      \item[(4.)] Every non-zero proper two-sided ideal $I$ in $R$ is of the form $\bigcap_i \mf_i^{k_i}$ for some maximal two-sided ideals $\mf_1,\dots,\mf_n$ and positive integers $k_1,\dots,k_n$.
  \end{enumerate}
\end{lemm}

\begin{proof}
  \noindent (1.) Let $p$ be a non-zero element of $I \cap \DC[t;\tau]$ of smallest degree, and note that $p$ is non-constant because $I$ is proper. In addition, $p$ has non-zero constant term because otherwise $pt^{-1}\in I \cap \DC[t;\tau]$ would be a polynomial of lower degree. Since $\DC$ is a division ring, we can take $1$ as the constant term. Now, $p$ commutes with $t$ because $p$ and $tpt^{-1}$ have constant term $1$ and $(p-tpt^{-1})t^{-1}\in I\cap \DC[t;\tau]$ has lower degree than $p$, and thus it must be zero. Similarly, $p$ commutes with every element of $\DC$, and we deduce that $p\in Z(R)$, i.e., $p\in S$.
  
  Since $\DC$ is a division ring and $\tau$ is an automorphism, we can divide polynomials in $\DC[t;\tau]$ (cf. \cite{BK2000}*{3.2.6}) and hence, given any other $p'\in I$ and an integer $k$ such that $t^kp'\in I \cap \DC[t;\tau]$ we have $t^kp' = qp+r$ for some $q,r\in \DC[t;\tau]$ with $\deg(r)<\deg(p)$. Since $r\in I\cap \DC[t;\tau]$,
  $r$ must be zero, and thus $p'=(t^{-k}q)p$. Therefore, $I = Rp$ as claimed. 
  
   Assume now that $I$ is not maximal and let $J$ be a proper two-sided ideal properly containing $I$. Then $J = Rq$ for some $q\in S$ as above, and hence $p = rq$ for some $r\in R$. Since $p,q$ have non-zero constant term, a comparison of degrees shows that $r\in \DC[t;\tau]$ and $0<\deg(r),\deg(q)<\deg(p)$. 
   In addition, for every $q'\in R$, we have
$$
  q'rq = q'p = pq' = rqq' = rq'q 
$$
from where $q'r = rq'$ since $R$ is a domain. Thus, $r\in S$ and $p = rq$ is a product of non-constant polynomials in $S$ of lower degree and therefore $p$ is not irreducible in $S$. 

Conversely, assume that $p$ is not irreducible in $S$, and let $r,q\in S$ be non-constant with $p = rq$. Note in particular that  $0<\deg(r),\deg(q)<\deg(p)$, and that $q$ must have non-zero constant term.  
Therefore $J = Rq$ is a proper two-sided ideal of $R$ with $I \subsetneq J$ (since $\deg(q)<\deg(p)$ and both have non-zero constant term) and hence $I$ is not maximal.  
\smallskip

\noindent (2.) By (1.), every maximal ideal $\mf$ is generated by an irreducible element $p\in S$ with non-zero constant term. Conversely, if $q\in S$ is irreducible with non-zero constant term we can check as in the final step of (1.) that $Rq$ is a maximal two-sided ideal of $R$. 

Let $\mf_1, \mf_2$ be maximal two-sided ideals in $R$, and assume $\mf_i = Rp_i$ for irreducible $p_1, p_2\in S$ with non-zero constant term. If $\mf_1 = \mf_2$, then $p_2 = r_1p_1$ and $p_1 = r_2p_2$ for some $r_1,r_2\in R$. Reasoning as above we can show that $r_i\in S$, and since $R$ is a domain, we obtain from $p_2 = r_1r_2p_2$ that they are units of $\DC$. Thus $p_1$ and $p_2$ are associates in $S$. Conversely, distinct maximals have non-associate generators, and the correspondence is bijective.

Now, let $Z(R) = K^{\tau}[s^{\pm 1}]$ with $s= ba^kt^{km}$ as in \cref{lemm:Laurent_center}. Then $S = K^{\tau}[s]$ and, since $Z(R)$ is again a Laurent polynomial ring over a field, the same argument above shows that the correspondence between maximal ideals of $Z(R)$ and irreducibles in $S$ with non-zero constant term (up to association) is bijective. Therefore, we have a bijection that sends the maximal two-sided ideal $\mf = Rp$ to the maximal ideal $\mf_Z = Z(R)p$. Moreover, $\mf_Z = \mf\cap Z(R)$, because if $x= rp\in \mf\cap Z(R)$ for some $r\in R$, then the same argument in (1.) shows that $r\in Z(R)$, and hence $x\in \mf_Z$. The other containment is clear. 
\smallskip

\noindent (3.)  Assume that $\mf_i = Rp_i$ for $p_i\in S$ irreducible with non-zero constant term. On the one hand, note that $\mf_1^{k_1}\dots \mf_n^{k_n} = Rp_1^{k_1}\dots p_n^{k_n} \subseteq \cap_i\mf_i^{k_i}$. On the other hand, let $q\in S$ with non-zero constant term be such that $\cap_i\mf_i^{k_i} = Rq$. In particular, $q = r_1p_1^{k_1} =\dots =r_np_n^{k_n}$ for some $r_1,\dots,r_n\in R$ and as above we can deduce that $r_i\in S$. But $S$ is an ordinary polynomial ring over a field, hence a unique factorization domain, and $p_1,\dots,p_n$ are non-associate irreducibles in $S$. Thus, necessarily $p_1^{k_1}\dots p_n^{k_n}$ divides $q$ in $S$, and therefore we also have $\cap_i\mf_i^{k_i} \subseteq \mf_1^{k_1}\dots\mf_n^{k_n}$.
\smallskip

\noindent (4.) If $I = Rp$ for some $p\in S$ with non-zero constant term and $p = up_1^{k_1}\dots p_n^{k_n}$ in $S$ with $p_i\in S$ non-associate irreducibles and $u$ a unit, then $p_i$ has non-zero constant term and we deduce from the proof of (3.) that $I = \cap_i \mf_i^{k_i}$ where $\mf_i = Rp_i$.
\end{proof}

Before getting to the description of $\mathbb P(R)$, we obtain a partial picture by analysing quotient rings $R/\mf^n$ for maximal two-sided ideals $\mf$. Observe that if $I$ is a non-zero proper two-sided ideal generated by a central element $q\in S$ with non-zero constant term, then $R/I$ is a left $\DC$-module with $\dim_{\DC}(R/I) = \deg(q)$ and with a natural $\DC$-basis $\{1+I,\dots, t^{\deg(q)-1}+I\}$. 

\begin{prop} \label{prop:quotient_by_maximal}
  For any maximal two-sided ideal $\mf$ of $R$ and positive integer $n$, there are exactly $n$ extreme points on $\mathbb P(R/\mf^n)$. Moreover, $\mathbb P(R/\mf^n) = \mathbb P_{\reg}(R/\mf^n)$ and the natural embedding $\varphi_n: Z(R)/(\mf^n\cap Z(R))\to R/\mf^{n}$ gives a bijection $\varphi_n^{\sharp}:\mathbb P(R/\mf^n) \to \mathbb P(Z(R)/(\mf^n\cap Z(R)))$. 
\end{prop}

\begin{proof}
By \cref{lemm:ideals_Laurent}(1.), there exists $p\in S$ with non-zero constant term and irreducible over $S$ such that $\mf = Rp$. In particular, this implies that $\mf_Z = \mf\cap Z(R) = Z(R)p$ is also a maximal 
ideal of $Z(R)$ by \cref{lemm:ideals_Laurent}(2.). 

For $n= 1$, $R/\mf$ is simple because $\mf$ is maximal. It is also left (and hence right; see \cite{GW2004}*{Corollary~4.18}) artinian because every descending chain of left ideals of $R/\mf$ is also a descending chain of left $\DC$-modules, and since $R/\mf$ is finite $\DC$-dimensional, it must stabilize. Thus, $R/\mf$ is
isomorphic to a matrix ring over a division ring, and hence it has only one rank function $\rk$. In this case $Z(R)/\mf_Z$ is a field and hence its unique rank function must then coincide with $\varphi_1^{\sharp}(\rk)$.

For $n\ge 2$, $R/\mf^n$ is a left artinian ring by the previous argument with a unique maximal two-sided ideal $\mf/\mf^n = (p+\mf^n)$. Since $p+\mf^n$ is central nilpotent (with order of nilpotency $n$), we have $p+\mf^n\subseteq J(R/\mf^n)$ and hence by maximality $J(R/\mf^n) = (p+\mf^n)$. 
Therefore, $J(R/\mf^n)$ is generated by a central element and  $(R/\mf^n)/J(R/\mf^n) = (R/\mf^n)/(\mf/\mf^n) \cong R/\mf$ is simple.  

Summing up, $R/\mf^n$ is a left artinian primary ring whose Jacobson radical is generated by a central element of order of nilpotency $n$. By \cref{coro:primary_art}, every rank function on $R/\mf^n$ is determined by its values on $p^i+\mf^n$ for $1\le i\le n-1$, and there are exactly $n$ extreme rank functions $\rk_{R/\mf^n,j}$ on $\mathbb P(R/\mf^n)$, defined by 
$$
  \rk_{R/\mf^n,j}(p^i+\mf^n) = \begin{cases}
      \frac{j-i}{j} & \textrm{if}~ i\leq j \\
                                 0 & \textrm{otherwise}.
  \end{cases}
$$
Now, fix $1\le j\le n$. Any element $q+\mf^n\in R/\mf^n$ gives rise to an endomorphism of the left $\DC$-module $R/\mf^j$ given by right multiplication by $q$. If $\deg(p) = l$, then $R/\mf^j$ has $\DC$-dimension $jl$ as noticed before and hence, fixing any basis of $R/\mf^j$, we have a ring homomorphism $\psi: R/\mf^n\rightarrow \End_{\DC}(R/\mf^j)\cong \Mat_{jl}(\DC)$ (recall that the endomorphisms act on the right). Therefore, we can define rank functions $\rk_{R/\mf^n,j}' = \psi^{\sharp}(\frac{1}{jl}\rk_{\DC})$ on $R/\mf^n$ and, although the latter isomorphism depends on the choice of the basis, the rank does not, since it is invariant under multiplication by invertible matrices. In particular, the image of $p^i+\mf^n$ with respect to the basis
$$
\{1+\mf^j,\dots,t^{l-1}+\mf^j, p+\mf^j,\dots,t^{l-1}p+\mf^j,\dots,p^{j-1}+\mf^j,\dots,t^{l-1}p^{j-1}+\mf^j \}
$$
is the $jl\times jl$ matrix of normalized rank $\frac{j-i}{j}$ 
$$
 \begin{pmatrix}
    0 & I_{(j-i)l}\\
    0 & 0
 \end{pmatrix}
$$
when $i\le j$, and the zero matrix otherwise. Therefore, necessarily $\rk_{R/\mf^n,j} = \rk_{R/\mf^n,j}'$ and the extreme ranks are regular. Since $\mathbb P_{\reg}(R/\mf^n)$ is convex, every rank is regular.

Finally, since $Z(R)$ is also a Laurent polynomial ring and $\mf_Z^n = \mf^n\cap Z(R)$, we have the same description of $\mathbb P(Z(R)/\mf_Z^n)$, i.e., every rank function is determined by its values on $p^i+\mf_Z^n$ and there are $n$ extreme points $\rk_{Z(R)/\mf_Z^n,j}$ defined as above on these elements. Since $\varphi_n$ sends $p^i+\mf_Z^n$ to $p^i+\mf^n$, the ranks $\varphi_n^{\sharp}(\rk_{R/\mf^n,j})$ and $\rk_{Z(R)/\mf_Z^n,j}$ give them the same value, and hence they are equal. Since the extreme points go to the extreme points and $\varphi_n^{\sharp}$ preserves convex combinations, this finishes the proof.
\end{proof}

As a consequence of the previous classification, we deduce the following result.

\begin{coro} \label{coro:quotient_by_twosided}
  Let $I$ be a non-zero proper two-sided ideal of $R$. If $I= Rp$ with $p\in S$ with non-zero constant term and $p = up_1^{k_1}\dots p_n^{k_n}$ is a factorization of $p$ into non-associate irreducibles of $S$, then every rank function on $R/I$ is determined by its values on $p_i^{j}+I$ for $i=1,\dots,n$, $j=1,\dots,k_i$.
\end{coro}

\begin{proof}
  We noticed during the proof of \cref{prop:quotient_by_maximal} that the statement holds when $I$ is a power of a maximal ideal. For the general case, as in the proof of  \cref{lemm:ideals_Laurent}(4.), $I = \cap_i \mf_i^{k_i}$, where $\mf_i = Rp_i$ are different maximal ideals of $R$. Since powers of different maximal two-sided ideals are comaximal, we have a ring isomorphism $\varphi:R/I\to \prod_i R/\mf_i^{k_i}$ by the Chinese Remainder Theorem. Take a rank function $\rk \in \mathbb P(R/I)$, and assume that $\rk = \varphi^{\sharp}(\rk')$ for $\rk'\in \mathbb P(\prod_i R/\mf_i^{k_i})$.
  
  In view of \cref{prop:directproduct}, if $\pi_i:\prod_i R/\mf_i^{k_i}\to R/\mf_i^{k_i}$ denotes the $i$-th projection, $\rk'$ can be uniquely written as a convex combination
  $$
   \rk' = \lambda_1\pi_1^{\sharp}(\rk_1)+\dots+ \lambda_n\pi_n^{\sharp}(\rk_n)
  $$
  for some $\rk_i\in \mathbb P(R/\mf_i^{k_i})$. Now, by comaximality,  $p_i+\mf_s^{k_s}$ is a unit in $R/\mf_s^{k_s}$ for every $s\ne i$, and therefore $\rk_s(p_i^j+\mf_s^{k_s}) = 1$ for every $s\ne i$ and every $j\ge 1$. Thus,
  $$
  \rk(p_i^j+I) = \rk'((p_i^j+\mf_s^{k_s})_s) = \lambda_i\rk_i(p_i^j+\mf_i^{k_i})+(1-\lambda_i).
  $$
   As a consequence, if we know the values of $\rk$ on the previous elements, then on the one hand we can compute the coefficients $\lambda_i$ via
  $$
   \rk(p_i^{k_i}+I) = \lambda_i\rk_i(p_i^{k_i}+\mf_i^{k_i})+(1-\lambda_i) = 1-\lambda_i
  $$
  and, on the other hand, for every $\lambda_i>0$, we can compute the values of $\rk_i$ on $p_i^j+\mf_i^{k_i}$ for $j=1,\dots,k_i-1$. As in the proof of \cref{prop:quotient_by_maximal}, this information determines $\rk_i$ uniquely, so we have completely determined $\rk$.
\end{proof}

We have now enough information to determine $\mathbb P(R)$ when $\tau$ has finite inner order. Let $\mf$ be any maximal two-sided ideal of $R$. The ring homomorphisms $\pi_{\mf, n}: R\to R/\mf^n$ for $n\ge 1$ allow us to define rank functions on $R$. In particular, if $\mf = Rp$ for some irreducible $p\in S$ with non-zero constant term, then for any positive integer $k$ we have a rank function $\rk_{\mf,k}$, which in the notation of \cref{prop:quotient_by_maximal} can be written as $\pi_{\mf, k}^{\sharp}(\rk_{R/\mf^k,k})$, satisfying
$$
  \rk_{\mf,k}(p^i) = \begin{cases}
      \frac{k-i}{k} & \textrm{if}~ i\leq k \\
                                 0 & \textrm{otherwise}.
  \end{cases}
$$
Moreover, if $q\in S$ is irreducible with non-zero constant term and not associated to $p$, then $\nf = Rq$ defines a different maximal two-sided ideal by \cref{lemm:ideals_Laurent}(2.), and hence $\nf + \mf^n = R$ for every $n \ge 1$. Thus, $q$ becomes a unit in every quotient $R/\mf^n$ and therefore $\rk_{\mf,k}(q^i) = 1$.
In terms of the associated Sylvester module rank functions $\dim_{\mf,k}$ we have then for a two-sided maximal ideal $\nf$, 
$$
\dim_{\mathfrak{m},k}(R/\nf^i) = \begin{cases}
\frac{i}{k} & \textrm{if}~ \nf = \mf ~\textrm{and}~ i\leq k \\
1 & \textrm{if}~ \nf = \mf ~\textrm{and}~ i > k \\
0 & \textrm{if}~ \nf \neq \mf
\end{cases}
$$
We can also define the rank function $\rk_0$ coming from its Ore quotient ring, whose associated module rank function $\dim_0$ satisfies $\dim_0(R/\nf^i) = 0$ for every maximal two-sided ideal $\nf$ and $i\ge 1$. 
We are going to prove that, in analogy to the case of Dedekind domains, these are the extreme rank functions on $R$. For this, we note first the following analog of \cref{lemm:countably_many_maximals_comm}, which is proved similarly invoking \cref{lemm:ideals_Laurent}(3.).

\begin{lemm} \label{lemm:countably_many_maximals_Laurent}
 Let $\dim$ be a Sylvester module rank function on $R=\DC[t^{\pm 1};\tau]$. There exist only countably many maximal two-sided ideals $\mf$ of $R$ such that $\dim(R/\mf^k)>0$ for some $k\ge 1$. 
\end{lemm}

Secondly, from \cref{lemm:ideals_Laurent}(1.) we know that $R$ is almost simple, and we observed in \cref{sect:simple_noetherian} that $\lkdim(R) = 1$. Hence, we can use \cref{prop:almost_simple} to deduce in the next proposition that the rank of a finite $\DC$-dimensional $R$-module $M$ is determined by its $Z(R)$-torsion submodule
 $$
 t_{Z(R)}(M) = \{m\in M: cm = 0 \textrm{ for some non-zero } c\in Z(R) \}.
 $$
Note that in the language of \cite{GW2004}*{page~83} this would be called $Z(R)\backslash\{0\}$-torsion submodule, and it is actually a submodule because $R$ is a domain and $Z(R)$ is commutative.

\begin{prop} \label{prop:central_torsion_sub}
If $M$ is a finitely generated $R$-module with $\dim_\DC(M)<\infty$, then, for any Sylvester module rank function $\dim \in \mathbb P(R)$, we have $\dim(M) = \dim(t_{Z(R)}(M))$. 
\end{prop}

\begin{proof}

Denote $N = t_{Z(R)}(M)$ and let $\dim$ be a Sylvester module rank function on $R$. If $N = 0$, then $M$ is $Z(R)$-torsionfree of finite length (since it is finite $\DC$-dimensional) and hence $\dim(M) = 0 = \dim(N)$ by \cref{prop:almost_simple}. Assume now that $N$ is non-zero. Since $R$ is noetherian, $N$ is also finitely generated, say, by $\{m_1,\dots,m_n\}$. For every $k$, we can take a non-zero $p_{m_k}\in Z(R)\cap \DC[t;\tau]$ with non-zero constant term that annihilates $m_k$ and observe that, if $p = \prod_k p_{m_k}$, then $N = \{m\in M: pm = 0\}$. Since $p\in Z(R)$, $pM$ is a submodule of $M$, and we claim that $M = N\oplus pM$.

Indeed, their intersection $N\cap pM$ is trivial, because if $pm\in N$, then $p^2m=0$, from where $m\in N$ (by definition, since $0\ne p^2\in Z(R)$), and therefore $pm=0$. Now, from the exact sequence (of $\DC$-modules) $0\to N \to M \to pM \to 0$, where the homomorphism $M\to pM$ is given by left multiplication by $p$, we obtain that $\dim_\DC(M) = \dim_\DC(N)+\dim_\DC(pM)$, and hence, since from the second isomorphism theorem and additivity of $\dim_{\DC}$,
$$ \dim_\DC(N+pM)+\dim_\DC(N\cap pM) = \dim_{\DC}(N)+\dim_{\DC}(pM) = \dim_{\DC}(M),$$
we deduce that $M = N \oplus pM$. Since $pM$ is $Z(R)$-torsionfree of finite length, \cref{prop:almost_simple} gives that $\dim(pM) = 0$, and therefore $\dim(M) = \dim(N)$ by (SMod2).
\end{proof}

\begin{theo} \label{theo:ranks_Laurent}
The extreme points on $\mathbb P(R)$ are precisely the rank functions $\dim_{\mf,k}$ and $\dim_0$ defined above, and any other rank function can be uniquely expressed as a (possibly infinite) convex combination of them. Thus, $\mathbb P(R) = \mathbb P_{\reg}(R)$. Moreover, the inclusion map $\iota: Z(R)\to R$ gives a bijection  $\iota^{\sharp}: \mathbb P(R)\to \mathbb P(Z(R))$.
\end{theo}
\begin{proof}
 Consider a Sylvester module rank function $\dim$ on $R$. Let $S_0$ be the set of two-sided maximal ideals $\mf$ for which there exists $k$ with $\dim(R/\mf^k) >0$, which is countable by \cref{lemm:countably_many_maximals_Laurent}. We are going to show that there exist non-negative real numbers $c_0, c_{\mf, k}$ satisfying
 $$
\dim = c_0\dim_0 + \sum_{\mf\in S_0} \sum_{k\in \mathbb{Z^+}} c_{\mf,k} \dim_{\mf,k} \textrm{~~ and ~~} c_0+\sum_{\mf\in S_0}\sum_{k\in \ZZ^+} c_{\mf,k} = 1.
$$
Set $b_{\mf,0}= \dim(R/\mf)$ and $b_{\mf,k}= \dim(R/\mf^{k+1})-\dim(R/\mf^k)$ for every $k\ge 1$. Since $\mf = Rp$ for some irreducible $p\in S = Z(R)\cap \DC[t;\tau]$ and $\mf^k = Rp^k$, \cref{lemm:positive_coefficients} implies that $b_{\mf,k}\ge b_{\mf,{k+1}}$ for every $k\ge 0$. Thus, if we impose equality of both expressions on the modules $R/\nf^i$ for maximal two-sided ideals $\nf$ and $i\ge 1$, we can reason as in \cref{theo:Dedekind} to show that the only possible solution is given by the non-negative coefficients 
$$
  c_{\mf,k} = k(b_{\mf,k-1}-b_{\mf,k}), k\ge 1
$$
and that $c_0=1-\sum_{\mf\in S_0}\sum_{k\in \ZZ^+} c_{\mf,k}$ is well-defined. Nevertheless, we do not know in principle that a rank function on $R$ is already determined by its values on those modules, so we still need to check that this is sufficient to claim equality. Define
$$
 \dim' := c_0\dim_0 + \sum_{\mf\in S_0} \sum_{k\in \mathbb{Z^+}} c_{\mf,k} \dim_{\mf,k}
$$
We have just seen that $\dim'$ is a Sylvester module rank function on $R$ that coincides with $\dim$ on the modules $R/\nf^k$ for every maximal two-sided ideal $\nf$ and $k\ge 1$. 

Take now any non-zero proper two-sided ideal $I$ such that $d_I = \dim(R/I) > 0$, and notice that $\dim'(R/I) = d_I$ (by using \cref{lemm:ideals_Laurent}(4.), the Chinese Remainder Theorem and that $\dim$ and $\dim'$ coincide on quotients by powers of maximal ideals). Now, $\frac{1}{d_I}\dim$ and $\frac{1}{d_I}\dim'$ define Sylvester module rank functions on $R/I$. Indeed, any finitely presented $R/I$-module is a finitely generated $R$-module with the natural operation, and hence finitely presented because $R$ is noetherian.
Thus, they satisfy (SMod1)-(SMod3) because $\dim$ and $\dim'$ do. But if $I = Rq$ for $q\in S$ with non-zero constant term and $q = uq_1^{k_1}\dots q_n^{k_n}$ for some non-associate irreducibles $q_i\in S$ and a unit $u\in S$, then rank functions on $R/I$ are determined by their value on $q_i^j+I$ for $i=1,\dots,n$, $j=1,\dots, k_i$ by \cref{coro:quotient_by_twosided}. Equivalently, if $\nf_i = Rq_i $, they are determined by their value on the modules $(R/I)/(R/I(q_i^j+I)) = (R/I)/(\nf_i^j/I) \cong R/\nf_i^j$ for $1\le i\le n$, $1\le j\le k_i$.
Therefore, by construction, $\frac{1}{d_I}\dim = \frac{1}{d_I}\dim'$ as rank functions on $R/I$.

Let $M$ be a finitely generated left $R$-module with $\dim_\DC(M)<\infty$. By \cref{prop:central_torsion_sub}, if $N = t_{Z(R)}(M)$, then $\dim (M) = \dim (N)$ and $\dim'(M) = \dim'(N)$. Thus, if $N = 0$ then $\dim(M) = 0 = \dim'(M)$. If $N$ is non-zero, then as in the proof of \cref{prop:central_torsion_sub} there exists $q\in S$ with non-zero constant term such that $N = \{m\in M: qm = 0\}$ and $M = N\oplus qM$. Define $I = Rq$, and note that $N$ is a finitely presented $R/I$-module.
If $d_I=0$ and $N$ is a $k$-generated $R$-module, we obtain from the surjective $R$-homomorphism $(R/I)^k\rightarrow N$, (SMod2) and (SMod3)
that $\dim(N) = \dim'(N) = 0$. Otherwise, the preceding paragraph shows that $\dim(N) = d_I (\frac{1}{d_I}\dim(N)) = d_I (\frac{1}{d_I}\dim'(N)) = \dim'(N)$. Therefore, $\dim(M) = \dim'(M)$.

Finally, take any matrix $A$ over $R$. Since adding rows and columns of zeros do not change the rank, we can assume that $A$ is $n\times n$. Take $k$ such that $A(I_nt^{k})$ is a matrix over $\DC[t;\tau]$. Thus, there exist $n\times n$ invertible matrices $P$, $Q$ over $\DC[t;\tau]$ respectively, and a diagonal matrix $D$, such that $A = PDQ(I_nt^{-k})$. (cf. \cite{BK2000}*{Proposition~3.2.8 \& 3.3.2}). Thus, any Sylvester matrix rank function on $R$ is determined by its values on elements, or equivalently any Sylvester module rank function is determined by its values on the quotients $R/Rp$ for $p\in R$, which are either free if $p=0$ or have finite dimension over $\DC$. Since we have seen that $\dim$ and $\dim'$ coincide on these modules, $\dim = \dim'$.

Since $Z(R)$ is a Laurent polynomial ring over a field, we have the same classification of rank functions on $Z(R)$. In view of \cref{lemm:ideals_Laurent}(2.), for every maximal two-sided ideal $\mf$ and every positive integer $k$, we have an extreme point $\rk_{\mf_Z, k}$ in $\mathbb P(Z(R))$ where $\mf_Z = \mf\cap Z(R)$. Moreover, if we denote by $\pi_{\mf_Z,k}: Z(R)\to Z(R)/\mf_Z^k$ the canonical homomorphism, then from the definition of the extreme rank functions and using \cref{prop:quotient_by_maximal}, we have
$$
  \rk_{\mf_Z,k} = \pi_{\mf_Z,k}^{\sharp}(\rk_{Z(R)/\mf_Z^k,k}) = \pi_{\mf_Z,k}^{\sharp}\varphi_k^{\sharp}(\rk_{R/\mf^k,k}) = \iota^{\sharp}\pi_{\mf, k}^\sharp(\rk_{R/\mf^k,k}) = \iota^{\sharp}(\rk_{\mf,k})
$$
The other extreme point in $\mathbb P(Z(R))$ is $\rk_{Z,0}$, the one induced by the unique rank function on its field of fractions $\mathcal Q(Z(R))$. By uniqueness, the rank function on $\mathcal Q(Z(R))$ is the restriction of the unique rank function on $\mathcal Q_l(R)$, and hence $\rk_{Z,0} = \iota^{\sharp}(\rk_0)$. Thus, the map $\iota^{\sharp}$ sends the extreme points to the extreme points, and since it preserves (infinite) convex combinations it is bijective. This finishes the proof.
\end{proof}

Notice that the same analysis done for $\DC[t^{\pm 1};\tau]$ works for ordinary polynomials $\DC[t]$ over a division ring $\DC$, since $\DC[t]$ enjoys the same properties that we have used for skew Laurent polynomials. Namely, $\DC[t]$ is a noetherian domain with $\lkdim(\DC[t]) = 1$ in which every two-sided ideal is generated by an element in $Z(\DC[t]) = Z(\DC)[t]$ (cf. \cite{GW2004}*{Theorem~1.9, Theorem~15.17 \& Proposition~17.1(c)}). From the latter property we can reason as in \cref{lemm:ideals_Laurent} to directly see that, if $K = Z(\DC)$, there is a bijection between maximal two-sided ideals in $\DC[t]$ and maximal ideals in $K[t]$, sending $\mf = Rp$ for some non-zero $p\in K[t]$ to $\mf_Z = \mf\cap K[t] = K[t]p$, and we can also deduce that every non-zero proper two-sided ideal can be written as intersection (equivalently, product) of powers of maximal ideals. With this, the proofs of
\cref{prop:quotient_by_maximal}, \cref{coro:quotient_by_twosided}, \cref{lemm:countably_many_maximals_Laurent} \cref{prop:central_torsion_sub} and \cref{theo:ranks_Laurent} apply with the corresponding changes to this case. Thus, we have the following partial answer to \cref{Question} for simple artinian rings.

\begin{prop} \label{polynomials}
  Let $R$ be a simple artinian ring. The inclusion $\iota: Z(R)[t]\to R[t]$ defines a bijection $\iota^{\sharp}: \mathbb P(R[t])\to \mathbb P(Z(R)[t])$. In particular, every Sylvester rank function on $Z(R)[t]$ can be uniquely extended to a Sylvester rank function on $R[t]$.
\end{prop}

\begin{proof}
 Since $R$ is simple artinian, there exists a division ring $\DC$ and a positive integer $n$ such that $R\cong \Mat_n(\DC)$. The inclusion map $\iota$ factors through
 $$
   Z(R)[t]\to Z(\DC)[t] \hookrightarrow \DC[t] \to \Mat_n(\DC[t])\to R[t],
 $$
 where the first map is the extension of the isomorphism $Z(R)\cong Z(\DC)$, the third one is the diagonal embedding and the last one is induced by $\Mat_n(\DC[t])\cong \Mat_n(\DC)[t]$. Consequently, $\iota^{\sharp}$ factors as
 $$
  \mathbb P(R[t]) \to \mathbb P(\Mat_n(\DC[t])) \to \mathbb P(\DC[t]) \to   \mathbb P(Z(\DC)[t]) \to \mathbb P(Z(R)[t]).
 $$
The first and the last map are induced by ring isomorphisms and hence bijective. The second one is bijective by \cref{prop:Morita}, and the third one is bijective by the analog of \cref{theo:ranks_Laurent}.
\end{proof}

Unfortunately, we cannot expect the same result to hold for skew polynomial rings $\DC[t;\tau]$. For example, if $\tau$ has infinite inner order, then $Z(\DC[t;\tau]) = K^{\tau}$ for $K = Z(\DC)$, and thus $\mathbb P(Z(\DC[t;\tau])) = \{\dim_{K^{\tau}}\}$, while the natural maps $\DC[t;\tau]\rightarrow \mathcal Q_l(\DC[t;\tau])$ and $\DC[t;\tau]\rightarrow \DC$ define two different Sylvester rank functions on the skew  polynomial ring.

\end{document}